%% file: main.tex
\documentclass[11pt]{article}
\usepackage[utf8]{inputenc}
\usepackage{cite}

\title{On definite lattices bounded by a homology 3-sphere\\ and Yang-Mills instanton Floer theory}
\author{Christopher Scaduto}

\date{}

\usepackage[a4paper, total={6in, 8.5in}]{geometry}
\usepackage{graphicx}

\usepackage{times}

\usepackage{amssymb}
\usepackage{tikz-cd}
\usepackage{titling}
\usepackage{mathrsfs} 
\usepackage{booktabs}
\usepackage[all]{xy}
\usepackage{amsthm}
\usepackage{diagbox}
\usepackage{tabularx}
\usepackage{amscd}
\usepackage{caption}
\usepackage{nicefrac}
\usepackage{amsmath}
\usepackage{mathtools}
\usepackage{tikz}
\usepackage{mathabx}
\usepackage{dsfont}
\usepackage{lipsum}
\usepackage{mwe}
\usepackage{slashed}
\usepackage{rotating}
\usepackage{subcaption}
\usepackage[colorlinks,pagebackref,hypertexnames=false]{hyperref} \usepackage[alphabetic,backrefs,msc-links]{amsrefs}
\usepackage{epstopdf,pinlabel}
\definecolor{mint}{HTML}{239B56}
\hypersetup{
    colorlinks=true,
    linkcolor=blue,
    filecolor=magenta,      
    urlcolor=cyan,
    citecolor=mint
}
\usepackage{pgfplots}
\pgfplotsset{compat=newest}
\usetikzlibrary{shapes.geometric}

\definecolor{greenish}{rgb}{0.01, 0.75, 0.24}
\definecolor{blueish}{rgb}{0.0, 0.72, 0.92}
\definecolor{orangeish}{rgb}{1.0, 0.55, 0.0}

\newcolumntype{Y}{>{\centering\arraybackslash}X}

\usepackage{sectsty}

\sectionfont{\fontsize{12}{15}\selectfont}

\newcommand{\R}{\mathbb{R}}
\newcommand{\C}{\mathbb{C}}
\newcommand{\Z}{\mathbb{Z}}
\newcommand{\F}{\mathbb{F}}
\newcommand{\Q}{\mathbb{Q}}

\newcommand{\A}{\mathsf{A}}
\newcommand{\E}{\mathsf{E}}
\newcommand{\cp}{\mathbb{C}\mathbb{P}^2}
\newcommand{\cpbar}{\overline{\mathbb{C}\mathbb{P}^2}}
\newcommand{\Ds}{\mathsf{D}}

\newtheorem{theorem}{Theorem}[section]
\newtheorem{prop}[theorem]{Proposition}
\newtheorem{lemma}[theorem]{Lemma}

\newtheorem{conjecture}[theorem]{Conjecture}
\newtheorem{corollary}[theorem]{Corollary}

\newcommand{\Addresses}{{
  \bigskip
  \footnotesize
Christopher Scaduto, \textsc{Department of Mathematics, University of Miami, Coral Gables, FL USA}\par\nopagebreak
  \textit{E-mail address}: \texttt{cscaduto@miami.edu}
}}

\begin{document}

\maketitle 

\vspace{-0.5cm}

\begin{abstract}
Using instanton Floer theory, extending methods due to Fr\o yshov, we determine the definite lattices that arise from smooth 4-manifolds bounded by certain homology 3-spheres. For example, we show that for +1 surgery on the (2,5) torus knot, the only non-diagonal lattices that can occur are $E_8$ and the indecomposable unimodular definite lattice of rank 12, up to diagonal summands. We require that our 4-manifolds have no 2-torsion in their homology.
\end{abstract}

\section{Introduction}\label{sec:intro}
\input{introduction.tex}

\section{The inequalities}\label{sec:latticeterms}
\input{lattice.tex}

\section{Genus 1 applications}\label{sec:genus1}
\input{genus1.tex}

\section{Genus 2 applications}\label{sec:genus2}
\input{genus2.tex}

\section{More examples}\label{sec:more}
\input{more.tex}

\section{Relations for a circle times a surface}\label{sec:surface}
\input{surface.tex}

\input{zetatable.tex}

\section{Adapting Fr\o yshov's argument}\label{sec:proofs}
\input{proofsineq.tex}

\section{Alternative proofs}\label{sec:alt}
\input{anotherproof.tex}

\section{The lattice $E_7^2$}\label{sec:e72}
\input{e72.tex}

\vspace{0.25cm}

\Addresses

\end{document}

%% file: introduction.tex
%!TEX root = main.tex

Let $X$ be a smooth, closed and oriented 4-manifold. The intersection of 2-cycles defines the structure of a unimodular lattice on the free abelian group $H_2(X;\Z)/\text{Tor}$. Donaldson's celebrated Theorem A of \cite{d-connections} says that if this lattice is definite, then it is equivalent over the integers to a diagonal form $\langle \pm 1 \rangle^n$. Donaldson's original proof used instanton gauge theory, and alternative proofs were later given using Seiberg-Witten and Heegaard Floer theory \cite[Thm.9.1]{os}, in conjunction with a lattice-theoretic result due to Elkies \cite{elkies-1}.\\

For a given integer homology 3-sphere $Y$, which definite lattices arise as the intersection forms of smooth 4-manifolds with boundary $Y$? Donaldson's theorem may be viewed as the solution to this problem in the case of the 3-sphere. To date, there is only one result in which the set of definite lattices is determined and does not consist of only diagonal lattices: under the assumption that the 4-manifolds are simply-connected, Fr\o yshov showed in his PhD thesis \cite{froyshov-thesis} that the only non-diagonalizable definite lattices bounded by the Poincar\'{e} sphere are $-E_8\oplus \langle - 1 \rangle^n$. The proof uses instanton gauge theory, and no other proofs are yet available.\\

In this article we extend and reformulate some of Fr\o yshov's methods in \cite{froyshov-thesis,froyshov-inequality} to obtain further results in this direction. The central new application is the following.\\

\begin{theorem}\label{thm:genus2}
    Let $Y$ be an integer homology 3-sphere $\Z/2$-homology cobordant to $+1$ surgery on a knot with smooth 4-ball genus 2. If a smooth, compact, oriented and definite 4-manifold with no 2-torsion in its homology has boundary $Y$, then its intersection form is equivalent to one of
    \begin{equation*}
        \langle \pm 1\rangle^{n}\;\;\qquad E_8\oplus \langle +1\rangle^{n}\;\; \qquad  \Gamma_{12}\oplus \langle +1\rangle^n \label{eq:genus2}
    \end{equation*}
where $\Gamma_{12}$ is the unique indecomposable unimodular positive definite lattice of rank 12.\\
\end{theorem}

If a non-diagonal lattice in this list occurs, then $\langle -1 \rangle^n$ for $n\geqslant 0$ does not: if the former arises from $X_1$ and the latter from $X_2$, both with boundary $Y$, then the closed 4-manifold $X_1\cup \overline{X}_2$ has a non-diagonal form, contradicting Donaldson's Theorem A. An example realizing all the positive forms on the list is $+1$ surgery on the $(2,5)$ torus knot, which is the Brieskorn sphere $-\Sigma(2,5,9)$.

\vspace{0.25cm}

\begin{corollary}\label{cor:2,5}
    If a smooth, compact, oriented and definite 4-manifold with no 2-torsion in its homology has boundary $-\Sigma(2,5,9)$, then its intersection form is equivalent to one of 
    \begin{equation*}
        \langle +1\rangle^{n}\;\; (n\geqslant 1), \qquad E_8\oplus \langle +1\rangle^{n}\;\;(n\geqslant 0), \qquad  \Gamma_{12}\oplus \langle +1\rangle^n \;\; (n\geqslant 0) \label{eq:genus2}
    \end{equation*}
    and all of these possibilities occur.
\end{corollary}

\vspace{0.25cm}

The realizations of these lattices are straightforward, except perhaps for the case of $E_8$; see e.g. \cite{golla-scaduto}. A slightly more general statement of Theorem \ref{thm:genus2} follows from Corollary \ref{cor:genus2} below. Theorem \ref{thm:genus2} may be viewed as the next installment of the following, which itself is a kind of successor to Donaldson's Theorem A cited above.

\vspace{0.25cm}

\begin{theorem}\label{thm:genus1} Let $Y$ be an integer homology 3-sphere $\Z/2$-homology cobordant to $+1$ surgery on a knot with smooth 4-ball genus 1. If a smooth, compact, oriented and definite 4-manifold with no 2-torsion in its homology has boundary $Y$, then its intersection form is equivalent to one of
\[
    \langle \pm 1 \rangle^n \qquad \;\;\;\; E_8\oplus \langle + 1 \rangle^n
\]
\end{theorem}

\vspace{0.25cm}

A corollary is a slight improvement of Fr\o yshov's theorem, obtained by applying the result to $+1$ surgery on the $(2,3)$ torus knot, the orientation-reversal of the Poincar\'{e} sphere $\Sigma(2,3,5)$:\\

\begin{corollary}[cf. \cite{froyshov-thesis}]\label{cor:2,3}
    If a smooth, compact, oriented and definite 4-manifold with no 2-torsion in its homology has boundary $-\Sigma(2,3,5)$, then its intersection form is equivalent to one of 
    \[
        \langle +1\rangle^{n}\;\;(n\geqslant 1),\qquad E_8\oplus \langle +1\rangle^{n}\;\;(n\geqslant 0),
    \]
    and all of these possiblities occur.
\end{corollary}

\vspace{0.25cm}

We give more examples in Section \ref{sec:more}. We expect the methods used to provide further applications. A good candidate to consider next is $-\Sigma(3,4,11)$, which is $+1$ surgery on the $(3,4)$ torus knot of genus 3. In \cite{golla-scaduto} we show that this manifold bounds the unimodular lattices $\langle +1\rangle$, $E_8$, $\Gamma_{12}$, $E_7^2$ and $A_{15}$, the last two being the indecomposable positive definite unimodular lattices of ranks 14 and 15, respectively; we exhibit some of these in Section \ref{sec:more}.\\

A straightforward Mayer-Vietoris argument shows that the statements of both Theorems \ref{thm:genus2} and \ref{thm:genus1} hold for 3-manifolds that are not integer homology 3-spheres, as long as the lattice is assumed to be {\emph{unimodular}}. In joint work with Marco Golla \cite{golla-scaduto} we provide analogues of the above results for non-unimodular lattices.\\

Other than Donaldson's Theorem A and Fr\o yshov's work in instanton Floer theory, restrictions on the possible definite lattices bounded by a fixed homology 3-sphere have previously been established using Seiberg-Witten and Heegaard Floer theory. In particular, there is a fundamental inequality for both the Heegaard Floer $d$-invariant of Oszv\'{a}th and Szab\'{o} \cite[Thm.1.11]{os} and Fr\o yshov's Seiberg-Witten correction term \cite[Thm.4]{froyshov-monopole}. A lattice theoretic result of Elkies \cite{elkies-2} implies that if an integer homology 3-sphere has either of these invariants the same as that of the Poincar\'{e} sphere, then there are only 14 possible definite lattices that occur, up to diagonal summands; see Table \ref{tab:table1}. While our proofs of all results stated above depend only on instanton theory, we will see that for Theorem \ref{thm:genus1} these restrictions from other theories can replace some, but not all, of the instanton theoretic input of the argument. The same is true for Theorem \ref{thm:genus2}, as discussed at the end of Section \ref{sec:genus2}.\\

To prove Theorems \ref{thm:genus2} and \ref{thm:genus1}, we provide partial analogues of Fr\o yshov's instanton inequality from \cite{froyshov-inequality} in which the coefficients used are the integers modulo a power of $2$, with an emphasis on the cases of $2$ and $4$. The inequalities provide new lower bounds for the genus of an embedded surface in a smooth closed 4-manifold in terms of data from the intersection form. Part of the input for these inequalities are relations in the instanton Floer cohomology ring of a circle times a surface, taken with the coefficient rings $\Z/2^k$. We only prove the relevant relations for low genus and small $k$, which is more than what is needed for our applications.\\

Apart from the determination of the relations just mentioned, the proofs of the inequalities we use are straightforward adaptations of the characteristic zero case from \cite{froyshov-inequality}, as explained in Section \ref{sec:proofs}. We also digress in Section \ref{sec:oddchar} to discuss analogues of Fr\o yshov's inequality for odd characteristic coefficients. However, these other variations do not appear to be useful.\\

Fr\o yshov has announced in several public lectures over the years the construction of two homology cobordism invariants, denoted $q_2$ and $q_3$, defined using the second and third Stiefel-Whitney classes of the basepoint fibration in the context of mod two instanton Floer theory, in a fashion similar to his construction of the $h$-invariant of \cite{froyshov-equivariant}. We expect that the inequalities studied here are relevant to this framework. We rather indirectly touch upon these matters in Section \ref{sec:alt}, where we replace our first arguments with some using instanton Floer theory for homology 3-spheres.\\

\textbf{Outline.} In Section \ref{sec:latticeterms} we state the inequalities obtained from instanton theory, our main technical tools. The proofs of these, which are adaptations of Fr\o yshov's argument to the settings of $\Z/2^k$ coefficients, are presented in Section \ref{sec:proofs}. In Section \ref{sec:genus1} we prove Theorem \ref{thm:genus1} and Corollary \ref{cor:2,3}. In Section \ref{sec:genus2} we prove Theorem \ref{thm:genus2} and Corollary \ref{cor:2,5}. More examples are presented in Section \ref{sec:more}. In Section \ref{sec:surface} we prove some relations in the instanton cohomology of a circle times a surface.  An alternative proof of Corollary \ref{cor:2,5}, closer in spirit to Fr\o yshov's proof of Corollary \ref{cor:2,3} and emphasizing the role of instanton Floer homology for homology 3-spheres, is presented in Section \ref{sec:alt}. Finally, in Section \ref{sec:e72}, we discuss an example of a rank 14 definite unimodular lattice $E_7^2$ which illustrates the necessity of the mod 4 data used in the proof of Theorem \ref{thm:genus2}.\\

\textbf{Acknowledgments.} The author thanks Kim Fr{\o}yshov for his encouragement and several informative discussions. The work here owes a great debt to his foundational work in instanton homology. Thanks to Motoo Tange for being the first to inform the author that $-\Sigma(2,5,9)$ bounds $E_8$. The author also thanks Marco Golla, Ciprian Manolescu, Matt Stoffregen and Josh Greene for helpful correspondences. 
The author was supported by NSF grant DMS-1503100.

%% file: lattice.tex
%!TEX root = main.tex

In this section we state partial analogues of Fr\o yshov's instanton inequality from \cite{froyshov-inequality} when the coefficients used are the integers modulo certain powers of $2$. Our primary focus will be the case of $\Z/4$; we also discuss the case of $\Z/2$, which is most relevant to Sections \ref{sec:alt} and \ref{sec:e72}. In addition, we make one use of the case $\Z/8$. The proofs of the results in this section are presented in Section \ref{sec:proofs}. For context, we also recall Fr\o yshov's inequality of \cite{froyshov-inequality}. The reader interested in the applications may wish to skip this section and refer back when needed.\\

Let $\mathbb{V}_g$ denote the $\mathbb{Z}/4$-graded instanton cohomology of a circle times a surface of genus $g$ equipped with a $U(2)$-bundle having odd determinant line bundle. The 4D cobordism defined by a 2D pair of pants cobordism crossed with the surface induces a map $\mathbb{V}_g\otimes \mathbb{V}_g \to \mathbb{V}_g$ endowing $\mathbb{V}_g$ with the structure of an associative ring with unit. Mu\~{n}oz \cite{munoz} determined a presentation for this ring over $\Q$ which is recursive in the genus, and we will see later that $\mathbb{V}_g$ is torsion-free. There are two distinguished elements in $\mathbb{V}_g$, denoted $\alpha$ and $\beta$, of degrees $2$ and $0$ mod 4, respectively. Define
\begin{equation*}
  N^2_\alpha(g) \;:=\; \min\left\{ n\geqslant 1:\; \alpha^n \equiv 0 \in \overline{\mathbb{V}}_g\otimes \Z/2 \right\}
\end{equation*}
for $g\geqslant 1$, and $N_\alpha^2(0)=0$. Here $\overline{\mathbb{V}}_g$ denotes the quotient of $\mathbb{V}_g$ by relative Donaldson invariants involving $\mu$-classes of loops; see Sections \ref{sec:surface} and \ref{sec:muclasses} for more details. The element $\alpha$(mod 2) in $\mathbb{V}_g\otimes\Z/2$ may be defined using the second Stiefel-Whitney class of the basepoint fibration, while $\beta\in \mathbb{V}_g$ is defined using the first Pontryagin class of the basepoint fibration. Let $k$ be a power of $2$. Define $N_\beta^k(0)=0$, and for $g\geqslant 1$, set
\begin{equation*}
  N^k_\beta(g) \;:=\; \min\left\{ n\geqslant 1:\; \beta^n \equiv 0 \in \overline{\mathbb{V}}_g\otimes \Z/k \right\}.
\end{equation*}
In Section \ref{sec:surface} we will see that $\alpha^2 \equiv \beta \pmod 8$, and it follows, for example, that $2N_\beta^4(g)\geqslant N_\alpha^2(g)$. Further, $\beta^2-64$ is nilpotent in $\mathbb{V}_g$, and so $N_\alpha^2(g)$ is finite, as is $N^k_\beta(g)$ for $k$ a power of $2$ at most $64$. Our primary focus will be on the case $k=4$.\\

Given a definite lattice $\mathcal{L}$ we define a non-negative integer $m(\mathcal{L})$ as follows. For a subset $S\subset \mathcal{L}$ denote by $\text{Min}(S)$ the elements which have minimal absolute norm among elements in $S$. Note that $\text{Min}(S)$ is of even cardinality when it is not $\{0\}$ and when $S$ is closed under negation, for in this case multiplication by $-1$ acts freely. We call $w\in\mathcal{L}$ {\emph{extremal}} if it is of minimal absolute norm in its index two coset, i.e. $w\in \text{Min}(w+2\mathcal{L})$. If $\mathcal{L}=0$, set $m(\mathcal{L})=0$. Otherwise, define
\begin{equation}
  m(\mathcal{L}) \;:=\; \max\left\{ |w^2|-1:\; w\neq 0 \text{ extremal}, \; \tfrac{1}{2}\#\text{Min}(w+2\mathcal{L})\equiv 1 \text{ (mod 2)} \right\}. \label{eq:f}
\end{equation}
It is straightforward to show that $m(\mathcal{L})=0$ for a diagonal lattice. In many examples in the sequel we bound $m(\mathcal{L})$ from below, and in some cases compute it.

\vspace{0.25cm}

\begin{theorem}\label{thm:char4} Let X be a smooth, closed, oriented 4-manifold with no $2$-torsion in its homology and $b_2^+(X)= 1$. Let $\Sigma\subset X$ be a smooth, orientable and connected surface in $X$ of genus $g$ with self-intersection 1. Let $\mathcal{L}\subset  H^2(X;\Z)/{\rm{Tor}}$ be the unimodular negative definite lattice of vectors vanishing on $[\Sigma]$. Then we have the inequality
    \begin{equation}
         N_\beta^4(g)  \; \geqslant \; f_{4} (\mathcal{L}),\label{eq:ineqchar4}
    \end{equation}
where $f_{4}(\mathcal{L})$ is a non-negative integer invariant of the unimodular lattice $\mathcal{L}$ defined below in (\ref{eq:f4}), satisfying $f_4(\mathcal{L})\geqslant \lceil m(\mathcal{L})/2\rceil$, and which vanishes if and only if $\mathcal{L}$ is diagonalizable.
\end{theorem}

\vspace{0.25cm}

As mentioned in the introduction, the proof is an adaptation of the characteristic zero case in \cite{froyshov-inequality}. Replacing $\Z/4$ with other coefficient rings yields similar results, which we comment on below and at various points throughout the article. However, Theorem \ref{thm:char4} is all that is needed to prove Theorem \ref{thm:genus2}.\\

If $X$ is negative definite, the inequality above applies to $X\# \cp$ with the genus zero exceptional sphere; in this case, $\mathcal{L}$ is the lattice of $X$. The vanishing of the left side of the inequality forces $\mathcal{L}$ to be diagonal, implying Donaldson's diagonalization theorem \cite{d-connections} assuming that $X$ has no $2$-torsion in its homology. In fact, $m(\mathcal{L})$, which also vanishes if and only if $\mathcal{L}$ is diagonal (see Prop. \ref{prop:diag}), essentially appears in Fintushel and Stern's proof of Donaldson's theorem \cite{fsdiag}.\\

The effectiveness of the inequality in Theorem \ref{thm:char4} towards our applications comes from the determination of $N_\beta^4(g)$. In Section \ref{sec:surface} we give evidence that the relations $\alpha^g \equiv 0$ (mod 2) and $\beta^{\lceil g/2 \rceil}\equiv 0$ (mod 4) hold in $\mathbb{V}_g$ for all $g$. For our applications, we only need verify this for $g\leqslant 2$. %To this end:

\vspace{0.25cm}

\begin{prop}\label{prop:nilppartial}
    For $g\leqslant 128$ we have $N_\alpha^2(g)=g$ and $N_\beta^4(g)=\lceil g/2\rceil$.
\end{prop}

\vspace{0.25cm}

We will in fact reduce the verification of this proposition for general $g$ to an elementary arithmetic problem which we do not attempt to solve in this article. The threshold $g=128$ is insignificant, and is the extent to which we have verified the formulas with a computer.\\

A partial analogue of Theorem \ref{thm:char4} with $\beta$ replaced by $\alpha$, and with coefficients $\mathbb{Z}/4$ replaced by $\mathbb{Z}/2$, is obtained as follows. Below we will define a lattice invariant $f_2(\mathcal{L})$ which arises naturally when counting reducibles mod $2$ in instanton moduli spaces cut down by the divisor associated to the second Stiefel--Whitney class of the basepoint fibration. The invariant $f_2(\mathcal{L})$ satisfies
\begin{equation}\label{eq:sandwich}
	2f_4(\mathcal{L}) \geqslant  f_2(\mathcal{L})  \geqslant m(\mathcal{L}).
\end{equation}
Now assume that the hypotheses of Theorem \ref{thm:char4} hold. Then Theorem \ref{thm:char4}, the inequality \eqref{eq:sandwich}, and Proposition \ref{prop:nilppartial} together imply the inequality
\begin{equation}\label{eq:mod2speculation}
	g \geqslant f_2(\mathcal{L})
\end{equation}
if $g$ is even and at most $128$. If $g$ is odd we still have $g+1\geqslant f_2(\mathcal{L})$. However, our computations suggest the possibility that \eqref{eq:mod2speculation} is true for all $g$. In Section \ref{sec:proofs} we explain the issue with directly adapting the argument for Theorem \ref{thm:char4} to this case.\\

While the full generality of \eqref{eq:mod2speculation} is left open, we will see that the invariants $f_2(\mathcal{L})$ and $m(\mathcal{L})$ arise as sometimes more useful invariants than $f_4(\mathcal{L})$ in the setting of instanton homology with mod $2$ coefficients, as is explored in Sections \ref{sec:alt} and \ref{sec:e72}.\\

We now define the lattice terms $f_2(\mathcal{L})$ and $f_4(\mathcal{L})$. Let $\mathcal{L}$ be a definite unimodular lattice. Given $x,y\in\mathcal{L}$ write $x\cdot y\in \Z$ for their inner product, and $x^2$ for $x\cdot x$. For $w\in \mathcal{L}$ write $\mathcal{L}^w \subset \mathcal{L}$ for the sublattice of elements $x\in\mathcal{L}$ satisfying $w\cdot x\equiv 0$ (mod 2). Given $z\in\mathcal{L}$, define a linear form $L_z: \text{Sym}^\ast(\mathcal{L}) \to \Z$ by first letting $L_z(a_1\cdots a_m) =(z \cdot a_1)\cdots (z \cdot a_m)$ where each $a_i\in \mathcal{L}$, and then extending linearly over $\Z$. Next, define 
\begin{equation}
	f_2(\mathcal{L}) \; := \; \max\left\{ |w^2|-m-1: \;\; 2^{-m}\eta \equiv 1 \text{ (mod 2)} \right\} \; \in \Z_{\geqslant 0}\label{eq:f}
\end{equation}
where the maximum is over triples $(w,m,a)$ where $w\in\mathcal{L}$ is nonzero and extremal, $m\in\Z_{\geqslant 0}$, $a\in \text{Sym}^m (\mathcal{L}^w)$, and, as indicated in (\ref{eq:f}) above, $2^{-m}\eta(\mathcal{L},w,a,m)\equiv 1$ (mod 2), where
\begin{equation}
	\eta(\mathcal{L},w,a,m) \; := \; \frac{1}{2^{}}\sum_{z\in \text{Min}(w+2\mathcal{L})}(-1)^{((z+w)/2)^2} L_z(a).\label{eq:eta}
\end{equation}

In (\ref{eq:f}) we use the convention that $\max(\emptyset)=0$. The conditions $a\in \text{Sym}^m (\mathcal{L}^w)$ and $w\neq 0$ imply that $\eta(\mathcal{L},w,a,m)$ is an integer divisible by $2^m$. The signs appearing in $\eta$ do not actually matter for the definition of $f_2(\mathcal{L})$, but do matter for the definitions to follow. When $m=0$ we interpret $L_z(a)=1$; in this case we simply write $\eta(\mathcal{L},w)$. Note that when $\mathcal{L}$ is an \emph{even} lattice the signs appearing in $\eta$ are all positive. We remark that our definition of $\eta$ is essentially that of \cite{froyshov-equivariant} and one half of that in \cite{froyshov-inequality}, except that in those references, only $a=a_0^m$ is used. Note that $\eta(\mathcal{L},w)\equiv 1$ (mod 2) is equivalent to the condition that $\frac{1}{2}\#\text{Min}(w+2\mathcal{L})\equiv 1$ (mod 2), and thus $f_2(\mathcal{L})\geqslant m(\mathcal{L})$. We do not have an example for which $f_2(\mathcal{L})>m(\mathcal{L})$, but we include $f_2(\mathcal{L})$ in our discussions because whatever we can prove for $m(\mathcal{L})$ also holds for $f_2(\mathcal{L})$.\\

Moving on to the lattice term in the mod 4 setting, we define
\begin{equation}
	f_4(\mathcal{L}) \; := \; \max\left\{   \frac{|w^2|-m}{2}: \;\; 2^{-m_0}\eta \not\equiv 0 \text{ (mod 4)} \right\} \; \in \Z_{\geqslant 0}\label{eq:f4}
\end{equation}
where the maximum is over triples $(w,m,a)$ where $w\in\mathcal{L}$ is nonzero and extremal, $m\in\Z_{\geqslant 0}$ with $w^2\equiv m$ (mod 2), $a\in \text{Sym}^{m_0}(\mathcal{L}^w)\otimes \text{Sym}^{m_1} (\mathcal{L})$ with $m_0+m_1=m$, and $2^{-m_0}\eta(\mathcal{L},w,a,m)\not\equiv 0$ (mod 4), as is indicated in (\ref{eq:f4}). As claimed in \eqref{eq:sandwich}, we have
\begin{equation*}
    f_4(\mathcal{L})\geqslant \lceil f_2(\mathcal{L})/2\rceil.
\end{equation*}
This follows directly from the definitions if $f_2(\mathcal{L})=|w^2|-m-1$ for an extremal vector $w$ with $a\in\text{Sym}^m(\mathcal{L}^w)$ and $\eta(\mathcal{L},w,a,m)\equiv 1$ (mod 2) where $w^2-m$ is even. If instead $w^2-m$ is odd, we use that $\eta(\mathcal{L},w,va,1+m)\equiv \eta(\mathcal{L},w,a,m) \equiv 1$ (mod 2) for any vector $v\in \mathcal{L}$ with $v\cdot w$ odd.

\vspace{0.25cm}

\begin{prop}\label{prop:diag}
    Each of $m(\mathcal{L})$, $f_2(\mathcal{L})$, $f_4(\mathcal{L})$ vanish if and only if $\mathcal{L}$ is diagonalizable.
\end{prop}

\vspace{0.15cm}

\begin{proof}
    Assume for simplicity that $\mathcal{L}$ is positive definite, and write $\mathcal{L}=\langle +1 \rangle^n\oplus L$ where $L$ contains no vectors of square $1$. If $\mathcal{L}$ is non-diagonalizable, then $L\neq 0$. Let $w\in L$ be of minimal nonzero norm in $L$. Suppose $v\in w+2\mathcal{L}$ and $v\neq \pm w$. Without loss of generality suppose $v\cdot w \geqslant 0$. Then $(w-v)^2= w^2 - 2w\cdot v + v^2 \leqslant w^2 + v^2$. On the other hand, $w-v \in 2L-\{0\}$, and so $(w-v)^2 \geqslant 4w^2$ since $w$ is minimal in $L-\{0\}$. We obtain $v^2 \geqslant 3w^2$. We conclude that $v\not\in\text{Min}(w+2\mathcal{L})$ and $\text{Min}(w+2\mathcal{L})=\{w,-w\}$, and thus $m(\mathcal{L})\geqslant w^2-1\geqslant 1$, and them same holds for $f_2(\mathcal{L})$ and $f_4(\mathcal{L})$ by \eqref{eq:sandwich}. The converse may be proved by direct computation, or we can apply Theorem \ref{thm:char4} with $X=\cp\# k\overline{\cp}$ and $\Sigma$ the exceptional sphere in $\cp$ to obtain that $f_2(\mathcal{L})=f_4(\mathcal{L})=m(\mathcal{L})=0$ for the diagonal lattice $\mathcal{L}=\langle +1 \rangle^k$.
\end{proof}

\vspace{0.25cm}

For the proof of Theorem \ref{thm:genus1} we will also make use of one inequality which arises from instanton constructions in the setting of mod $8$ coefficients.

\vspace{0.25cm}

\begin{prop}\label{prop:mod8}
	Define $f_8(\mathcal{L})$ by replacing ``{\emph{mod $4$}}'' in the definition of $f_4(\mathcal{L})$ with ``{\emph{mod $8$}}''. Assume the hypotheses of Theorem \ref{thm:char4}, and that $g=1$. Then $f_8(\mathcal{L})\in \{0,1\}$.
\end{prop}

\vspace{0.25cm}

The proof is similar to that of Theorem \ref{thm:char4}, with the additional input that $N^8_\beta(g)=1$.\\

The lattice terms appearing above should be compared to the analogous term appearing in Fr\o yshov's inequality for the instanton $h$-invariant, which is defined in the setting of $\Q$-coefficients. We now recall his result. In fact, we will state a slightly more general result. We define for a definite unimodular lattice $\mathcal{L}$ the following quantity:
\begin{equation}
	e_0(\mathcal{L}) \; := \; \max\left\{  \left\lceil \frac{|w^2|-m}{4} \right\rceil: \;\; \eta \neq 0 \right\} \; \in \Z_{\geqslant 0}\label{eq:e}
\end{equation}
where the maximum is over triples $(w,m,a)$ where $w\in\mathcal{L}$ is extremal, $m\in\Z_{\geqslant 0}$, $w^2\equiv m$ (mod 2), $a\in \text{Sym}^m(\mathcal{L})$, and $\eta(\mathcal{L},w,a,m)\neq 0$, as is abbreviated in (\ref{eq:e}). From the definitions we have
\begin{equation*}
    e_0(\mathcal{L})\; \geqslant \; \lceil f_4(\mathcal{L})/2 \rceil \; \geqslant\; \lceil f_2(\mathcal{L})/4\rceil.
\end{equation*}
Denote by $h(Y)$ Fr\o yshov's instanton $h$-invariant defined in \cite{froyshov-equivariant}. We next define
\begin{equation*}
  N^0_\beta(g) \;:=\; \min\left\{ n\geqslant 1:\; (\beta^2 - 64)^{n} = 0 \in \overline{\mathbb{V}}_g\otimes \Q \right\}
\end{equation*}
for $g\geqslant 1$ and $N^0_\beta(0)=0$. The computation $N_\beta^0(g)\leqslant \lceil g/2 \rceil$ due to Mu\~{n}oz \cite[Prop.20]{munoz} is used in Fr\o yshov's inequality, and determines the left hand sum in the following.

\vspace{0.25cm}

\begin{theorem}[cf. \cite{froyshov-inequality} Thm.2]\label{thm:char0} Let X be a smooth, compact, oriented 4-manifold with homology 3-sphere boundary $Y$ and $b_2^+(X)=n\geqslant 1$. Let $\Sigma_i\subset X$ for $1\leqslant i \leqslant n$ be smooth, orientable, connected surfaces in $X$ of genus $g_i$ with $\Sigma_i\cdot\Sigma_i=1$ which are pairwise disjoint. Denote by $\mathcal{L}\subset  H^2(X;\Z)/\text{\emph{Tor}}$ the unimodular lattice of vectors vanishing on the classes $[\Sigma_i]$. Then
    \begin{equation}
         h(Y) +\sum_{i=1}^n \lceil g_i/2 \rceil  \; \geqslant \; e_0(\mathcal{L}).\label{eq:ineqchar0}
    \end{equation}
\end{theorem}

\vspace{0.25cm}

We have lifted the restriction in \cite{froyshov-inequality} that all but one of the surfaces have genus 1. This follows from a minor technical improvement of the proof, which uses the existence of a perfect Morse function on the moduli space of projectively flat $U(2)$ connections on a surface with fixed odd determinant. This is explained in Section \ref{sec:proofs}.\\

Each of the lattice terms defined above arises from adapting the proof of Fr\o yshov's inequality; each such adaptation has a choice of coefficient ring, a corresponding relation in the instanton cohomology ring of a circle times a surface, and a possible assumption on the torsion group $\mathcal{T}\subset H^2(X;\Z)$ of the 4-manifold. We summarize the expected scheme for some of the cases above:

\vspace{0.25cm}

\begin{center}
\renewcommand{\arraystretch}{1.5}
\begin{tabular}{ c c c c }
    Lattice term \quad & Coefficients\quad  & Relation \quad & Torsion assumption\\
    \hline
  $e_0(\mathcal{L})$ & $\Q$  & $(\beta^2-64)^{\lceil g/2\rceil} = 0$  & none\\
  $f_2(\mathcal{L})$ & $\Z/2$ & $\alpha^g \equiv 0$ (mod 2) & $2\nmid \#\mathcal{T}$ \\
  $f_4(\mathcal{L})$ & $\Z/4$ & $\beta^{\lceil g/2 \rceil} \equiv 0$ (mod 4) &  $2 \nmid \#\mathcal{T}$ \\
    $\lceil f_2(\mathcal{L})/2 \rceil$ & $\Z/4$ & $\beta^{\lceil g/2 \rceil} \equiv 0$ (mod 4) &  $4 \nmid \#\mathcal{T}$ \\
\end{tabular}

\end{center}

\vspace{0.25cm}

The relations are to be understood within $\overline{\mathbb{V}}_g$, although we expect the mod 2 and mod 4 relations, which as listed are only verified for $g\leqslant 128$ in this paper, to hold in $\mathbb{V}_g$. The first row corresponds to Theorem \ref{thm:char0}, the second row to inequality \eqref{eq:mod2speculation} (which is established for $g$ even and $g\leqslant 128$), and the third row to Theorem \ref{thm:char4}. The fourth row is the result of slightly relaxing the torsion assumption in the proof of Theorem \ref{thm:char4}. However, we will make no use of it and will not mention it further.\\

We have only included in our discussion the variations of Fr\o yshov's inequality we have found useful for our applications. However, the proof of Theorem \ref{thm:char0} is easily adapted to any coefficient ring. We discuss this to some extent in Section \ref{sec:oddchar}.\\

In Section \ref{sec:e72} we show that the indecomposable unimodular positive definite lattice of rank 14 has $f_4(\mathcal{L})=2$, while $e_0(\mathcal{L})=1$ and $f_2(\mathcal{L})=2$. This example shows the necessity of the inequality associated to mod 4 coefficients in proving Theorem \ref{thm:genus2}.

%% file: genus1.tex
%!TEX root = main.tex

In this section we prove Theorem \ref{thm:genus1} and Corollary \ref{cor:2,3} assuming the results of Section \ref{sec:latticeterms}, and using Heegaard Floer $d$-invariants. Next section we will show that these results can be proved without Heegaard Floer theory, using only our instanton obstructions. We begin with a corollary of our inequalities that follows \cite[Cor. 1]{froyshov-inequality}.\\

For a knot $K$ in an integer homology 3-sphere $Y_0$, we define $g_{4,2}(K)$ to be the minimum over all $g\geqslant 0$ such that there exists a $\Z/2$-homology 4-ball $W$ with $\partial W=Y_0$ and an oriented, genus $g$ surface $\Sigma$ smoothly embedded in $W$ with $\partial \Sigma = K$. If no such data exists, we set $g_{4,2}(K)=\infty$. If $K$ is a knot in the 3-sphere, note $g_{4,2}(K)\leqslant g_4(K)$, the latter being the smooth 4-ball genus of $K$.

\vspace{0.25cm}

\begin{corollary}\label{cor:link}
    Let $Y$ be an integer homology 3-sphere resulting from $(-1)$-surgery on a knot $K$ in an integer homology 3-sphere. Suppose $Y$ bounds a smooth, compact, oriented 4-manifold $X$ with no 2-torsion in its homology and negative definite intersection form $\mathcal{L}$. If $g_{4,2}(K)\leqslant 128$, then
    \[
        f_4(\mathcal{L})\leqslant \lceil g_{4,2}(K)/2\rceil.
    \]
Furthermore, if $g_{4,2}(K)=1$, then $f_8(\mathcal{L})\in \{0,1\}$. 
\end{corollary}

\vspace{0.25cm}

To obtain the corollary, let $Z$ be the orientation-reversal of the negative definite surgery cobordism from $Y_0$ to $Y$. Then apply Theorem \ref{thm:char4} to the closed 4-manifold $X\cup_{Y} Z \cup_{Y_0}  W $, which has a surface of self-intersection $1$ formed by capping off the component of a surface $\Sigma\subset W$ bounded by $K$ as above with a disk from the 2-handle of the surgery cobordism $Z$. Proposition \ref{prop:nilppartial} determines the left hand side of (\ref{eq:ineqchar4}) for $g\leqslant 128$, and the inequality for $f_4(\mathcal{L})$ follows. Apply Proposition \ref{prop:mod8} to obtain the last statement regarding $f_8(\mathcal{L})$.\\

We recall some basic notions from the theory of lattices. Let us call a definite lattice {\emph{reduced}} if there are no elements of squared norm $\pm 1$. A {\emph{root}} in a reduced definite lattice $\mathcal{L}$ is an element with square $\pm 2$. A {\emph{root lattice}} is a reduced positive definite lattice generated by its roots. Examples are $\A_n,\Ds_n,\E_6,\E_7$ and $\E_8$, each associated to a Dynkin diagram:
\begin{center}
\begin{tabular}{ll}

\begin{tikzpicture}
	
	\draw (0,0) -- (1.3,0);
	\draw (1.7, 0) -- (2.5,0);
	
	\draw[fill=white] (0,0) circle (.1);
	\draw[fill=white] (.5,0) circle (.1);
	\draw[fill=white] (1,0) circle (.1);
	\draw[fill=white] (2,0) circle (.1);
	\draw[fill=white] (2.5,0) circle (.1);
	
	\node at (1.5,0) {$\cdots$};
	\node at (-1,0) {$\A_n$};
	 
\end{tikzpicture}

&
\qquad

\begin{tikzpicture}

	\draw (0,0) -- (2,0);
	\draw (1,0) -- (1,.5);
	
	\draw[fill=white] (0,0) circle(.1);
	\draw[fill=white] (.5,0) circle(.1);
	\draw[fill=white] (1,0) circle(.1);
	\draw[fill=white] (1,.5) circle(.1);
	\draw[fill=white] (1.5,0) circle(.1);
	\draw[fill=white] (2,0) circle(.1);
	
	\node at (-1,0) {$\E_6$};
	 
\end{tikzpicture}

\\

\begin{tikzpicture}

	\draw (0,0) -- (1.3,0);
	\draw (1.7, 0) -- (2,0);
	\draw (2,0) -- (2.5,-.25);
	\draw (2,0) -- (2.5,.25);
	
	\draw[fill=white] (0,0) circle(.1);
	\draw[fill=white] (.5,0) circle(.1);
	\draw[fill=white] (1,0) circle(.1);
	\draw[fill=white] (2,0) circle(.1);
	\draw[fill=white] (2.5,-.25) circle(.1);
	\draw[fill=white] (2.5,.25) circle(.1);
	
	\node at (1.5,0) {$\cdots$};
	\node at (-1,0) {$\Ds_n$};
	 
\end{tikzpicture}

&
\qquad

\begin{tikzpicture}

	\draw (0,0) -- (2.5,0);
	\draw (1,0) -- (1,.5);
	
	\draw[fill=white] (0,0) circle(.1);
	\draw[fill=white] (.5,0) circle(.1);
	\draw[fill=white] (1,0) circle(.1);
	\draw[fill=white] (1,.5) circle(.1);
	\draw[fill=white] (1.5,0) circle(.1);
	\draw[fill=white] (2,0) circle(.1);
	\draw[fill=white] (2.5,0) circle(.1);
	
	\node at (-1,0) {$\E_7$};
	 
\end{tikzpicture}

\\

&
\qquad

\begin{tikzpicture}

	\draw (0,0) -- (3,0);
	\draw (1,0) -- (1,.5);
	
	\draw[fill=white] (0,0) circle(.1);
	\draw[fill=white] (.5,0) circle(.1);
	\draw[fill=white] (1,0) circle(.1);
	\draw[fill=white] (1,.5) circle(.1);
	\draw[fill=white] (1.5,0) circle(.1);
	\draw[fill=white] (2,0) circle(.1);
	\draw[fill=white] (2.5,0) circle(.1);
	\draw[fill=white] (3,0) circle(.1);
	
	\node at (-1,0) {$\E_8$};
	 
\end{tikzpicture}

\end{tabular}
\end{center}
The root lattice is obtained by taking as basis the vertices, each having square $2$; if two vertices are joined by an edge, their inner product is $-1$, and is otherwise $0$. For $\A_n$ we require $n\geqslant 1$, and for $\Ds_n$, $n\geqslant 4$. In each case, $n$ is the number of vertices, or the rank of the lattice. It is well-known that any positive definite root lattice can be written as a direct sum of these given lattices.\\

To simplify the notation below, we assume henceforth that $\mathcal{L}$ is a positive definite unimodular lattice. Any such lattice can be written as $\mathcal{L}=\langle +1\rangle^n\oplus L$ where $L$ is reduced and $n\geqslant 0$. We write $R\subset L$ for the root lattice generated by the roots of $L$, and also call $R$ the {\emph{root lattice of }}$\mathcal{L}$. In general, $\mathcal{L}$ is not determined by $R$, but it is common in many cases to notate $L$ by the data $R$, cf. \cite[Ch.16]{conwaysloane}. For example, we write $A_{15}$ for the rank 15 unimodular positive definite lattice whose root lattice $R$ is isomorphic to $\A_{15}$. For this reason we have used different fonts for unimodular lattices and root lattices, although $E_8=\E_8$. The presence of an ``$O$'' indicates an empty root lattice; for example, the lattice $O_{23}$, called the shorter Leech lattice, has no roots.

\vspace{0.25cm}

\begin{lemma}\label{lemma:indecomp}
    If $f_4(\mathcal{L})=1$ then the root lattice $R\subset \mathcal{L}$ is indecomposable.
\end{lemma}

\vspace{0.15cm}

\begin{proof} Write $\mathcal{L}=\langle +1\rangle^n\oplus L$ as above, so that $R\subset L$. Suppose $w\in R$ is extremal in $R$ and $w^2=4$. We first claim that $\text{Min}(w+2\mathcal{L})=\text{Min}(w+2R)$. Let $v\in \text{Min}(w+2\mathcal{L})$ with $v\not\in R$, and suppose without loss of generality that $w\cdot v\geqslant 0$. Then $(w-v)^2 = 4 - 2w\cdot v + v^2 \leqslant 4+ v^2$. On the other hand, $w-v\in 2(L-R)$ implies $(w-v)^2 \geqslant 4\cdot 3 = 12$, since $L-R$ has vectors only of square $\geqslant 3$. Thus $v^2 \geqslant 8$, contradicting the assumption that $v$ is extremal. This proves the claim.\\

Now suppose $R$ is decomposable, i.e. $R=R_1\oplus R_2$. Then there are $u\in R_1$ and $v\in R_2$ both of square $2$. Set $w=u+v\in R$, which has $w^2=4$. Then $\text{Min}(w+2\mathcal{L})=\{\pm u \pm v\}$ contains 4 elements, and $\eta(\mathcal{L},w)=2\not\equiv 0$ (mod 4). Thus $f_4(\mathcal{L})\geqslant w^2/2  =2$.
\end{proof}
\vspace{0.25cm}

The following lemma is not needed for what follows, but serves as a warmup for the next section. Furthermore, it will be used in Section \ref{sec:alt} to give an alternate proof of Corollary \ref{cor:2,3}.

\vspace{0.25cm}

\begin{lemma}\label{lemma:e8}
    If $m(\mathcal{L})=1$ and $R\subset\mathcal{L}$ is indecomposable, then $\mathcal{L}= E_8\oplus \langle +1 \rangle^n$ for some $n\geqslant 0$.
\end{lemma}

\vspace{0.15cm}

\begin{proof} 
We claim the map $\pi:R\otimes \Z/2\to L\otimes \Z/2$ induced by inclusion is an isomorphism. (This is essentially the proof of \cite[Lemma 4.3]{froyshov-thesis}.) Suppose it is not. Choose $w$ of minimal norm such that $[w]$ is not in the image of $\pi$. In particular, $w$ is extremal. Now suppose $v=w+2u$ is extremal with $v\cdot w \geqslant 0$ and $v\neq \pm w$. If $[u]\not\in \text{im}(\pi)$ then $2w^2\geqslant w^2-2w\cdot v + v^2 =(w-v)^2=4u^2 \geqslant 4w^2$, the last inequality by minimality of $w$. This is a contradiction, and so $[u]\in\text{im}(\pi)$. In particular, $[w\pm u]\not\in\text{im}(\pi)$. Then $(w\pm u)^2\geqslant w^2$ implies $2|w\cdot u|\leqslant u^2$. But $w^2=v^2=w^2+4w\cdot u+4u^2$ implies $|w\cdot u|=u^2$, whence $u=0$. It follows that $\text{Min}(w+2\mathcal{L})=\{w,-w\}$. Then $m(\mathcal{L})\geqslant w^2-1\geqslant 2$, since $w$ is not a root, contradicting our hypothesis on $m(\mathcal{L})$.\\

Thus $\pi$ is an isomorphism. In particular, $\text{rank}(R)=\text{rank}(L)$ and $\det(R)$ is odd. If $R$ is indecomposable, the latter condition implies that $R$ is either zero, $\mathsf{E}_6$, $\mathsf{E}_8$ or $\mathsf{A}_n$ for $n\geqslant 2$ even. That $m(\mathcal{L})=1$ when $L=R=E_8$ follows from direct computation, or by applying Corollary \ref{cor:link} to $+1$ surgery on the $(2,3)$ torus knot, which bounds $E_8$.\\

If $R$ is zero, so is $L$, since the ranks are equal. But then $\mathcal{L}$ is diagonal, contradicting $m(\mathcal{L})=1$. Next, suppose $R=\mathsf{A}_n$. A standard model of $\mathsf{A}_n$ is the sublattice of $\Z^{n+1}$ spanned by vectors whose coordinates add up to zero. Suppose $n\geqslant 3$, and let $w=(1,1,-1,-1,0,\ldots,0)\in\mathsf{A}_n$. Then $w$ is extremal in $\mathsf{A}_n$ with square 4, and $\text{Min}(w+2\mathcal{L})=\text{Min}(w+2\mathsf{A}_{n})$ consists of the $6$ vectors obtained from $w$ by permuting the two signs. Thus $\frac{1}{2}\#\text{Min}(w+2\mathcal{L})=3$, implying $m(\mathcal{L})\geqslant w^2-1= 3$. Finally, the cases $\mathsf{E}_6$ and $\mathsf{A}_2$ are ruled out by $\text{rank}(R)=\text{rank}(L)$; it is well-known that there are no unimodular, non-diagonal definite lattices of rank $<8$. 
\end{proof}

\vspace{0.25cm}

We next recall the fundamental inequality for the Heegaard Floer $d$-invariant of Oszv\'{a}th and Szab\'{o} \cite[Thm. 1.11]{os}. This states that if $Y$ is an integer homology 3-sphere, and $X$ is a smooth, negative definite 4-manifold bounded by $Y$, then for any characteristic vector $\xi\in H^2(X;\Z)/\text{Tor}$,
\begin{equation}
    d(Y) \; \geqslant \; \frac{1}{4}\left(b_2(X)-|\xi^2|\right).\label{eq:dinv}
\end{equation}
Recall that a characteristic vector $\xi$ is an element that satisfies $\xi \cdot x \equiv x^2$ (mod 2) for every $x$ in the lattice. It is classically known that the square of any characteristic vector is modulo 8 the rank of the lattice. Elkies showed in \cite{elkies-2} that, up to adding diagonal summands $\langle +1\rangle^n$, there are a finite number of positive definite unimodular lattices with no characteristic vectors of squared norm less than $n-8$, where $n$ is the rank of the lattice. There are in fact 14:
\begin{table}[h!]
  \centering
  \caption{Elkies' list}
\vspace{.1cm}
  \label{tab:table1}
\renewcommand{\arraystretch}{1.2}
  \begin{tabular}{ccccccccccccc}
    \hline
    $n$ & 8 & 12 & 14 & 15 & 16 & 17 & 18  & 19 & 20 & 21 & 22 & 23\\
    \hline
    & $E_8$ & $D_{12}$ & $E_7^2$ & $A_{15}$ & $D_8^2$ & $A_{11}E_6$ & $D_6^3$, $A_9^2$ & $A_7^2 D_5$ & $D_4^5$, $A_5^4$ & $A^7_3$ & $A_1^{22}$ & $O_{23}$  \\ 
    \hline
  \end{tabular}
\end{table}
Thus by (\ref{eq:dinv}), if a non-diagonal definite lattice is bounded by $Y$ with $d(Y)=-2$, as is the case for the orientation-reversal of the Poincar\'{e} homology 3-sphere, it must be one of these 14 lattices, possibly upon adding $\langle + 1 \rangle^n$. We remark that Seiberg-Witten theory can also be used make this reduction, as Fr\o yshov's monopole invariant (rescaled) also satisfies (\ref{eq:dinv}), see \cite[Thm. 4]{froyshov-monopole}. It is known that if $Y$ is $+1$-surgery on a knot of slice genus 1 we have $d(Y)\in\{0,-2\}$, see (\ref{eq:dgenus}). According to Elkies \cite{elkies-1}, if $d(Y)=0$, the only possible definite lattices that $Y$ can bound are diagonal.\\

We obtain the following, which, along with the observation that the statement is $\Z/2$-homology cobordism invariant (see Section \ref{sec:more}), implies Theorem \ref{thm:genus1}.

\vspace{0.25cm}

\begin{corollary}\label{cor:genus1}
Let $Y$ be an integer homology 3-sphere resulting from $(+1)$ surgery on a knot $K$ in an integer homology 3-sphere with $g_{4,2}(K) = 1$. If $X$ is a smooth, compact, oriented and definite 4-manifold bounded by $Y$ with non-diagonal lattice $\mathcal{L}$ and no 2-torsion in its homology, then $\mathcal{L}=\langle +1 \rangle^n\oplus E_8$ for some $n\geqslant 0$.
\end{corollary}

\vspace{0.15cm}

\begin{proof}
	From the above remarks regarding $d$-invariants and Elkies' result, the reduced part $L$ of $\mathcal{L}$ is among the 14 lattices in Table \ref{tab:table1}. By Corollary \ref{cor:link} we have $f_4(\mathcal{L})\leqslant 1$. We may assume $f_4(\mathcal{L})=1$, for otherwise $\mathcal{L}$ is diagonal. By Lemma \ref{lemma:indecomp}, $L$ must be one of $E_8$, $D_{12}=\Gamma_{12}$, $A_{15}$, or $O_{23}$. We must rule out the last 3 possibilities.\\

Suppse $L=A_{15}$. As in Lemma \ref{lemma:e8}, we take $w=(1,1,-1,-1,0,\ldots,0)\in \mathsf{A}_{15}$, which is extremal and has $w^2=4$, with $\eta(\mathcal{L},w)=3\not\equiv 0 \pmod 4$. Thus $f_4(\mathcal{L})\geqslant w^2/2=2$, ruling this possibility out.\\

Suppose $L=O_{23}$. Minimal vectors in $O_{23}$ have square $3$. Take any $w\in  O_{23}$ with $w^2=4$. Such vectors exist by inspecting the theta series of $O_{23}$, given in (7) of \cite[p.443]{conwaysloane}. Then $w$ is extremal and $\text{Min}(w+2\mathcal{L})=\{w,-w\}$, so again we are led to $f_4(\mathcal{L})\geqslant w^2/2 = 2$, eliminating $O_{23}$.\\

Finally, consider the case $L=D_{12}$. Here $f_4(\mathcal{L})=1$ (see Proposition \ref{prop:gamma}), and we use instead the constraint from Corollary \ref{cor:link} that $f_8(\mathcal{L})=1$. To this end we take $w=(1,1,1,1,0,\ldots,0)\in \mathsf{D}_{12}$ as our extremal vector. Then $\text{Min}(w+\mathcal{L})=\text{Min}(w+\mathsf{D}_{12})$ consists of the vectors of the form $(\pm 1, \pm 1, \pm 1, \pm 1, 0,\ldots, 0)$ where the number of signs is even. Thus $\eta(\mathcal{L},w)=\frac{1}{2}\# \text{Min}(w+\mathcal{L}) = \frac{1}{2}8=4\not\equiv 0 \pmod 8$. Thus $f_8(\mathcal{L})\geqslant w^2/2=2$, which rules out $D_{12}$ and completes the proof.
\end{proof}

\vspace{0.25cm}

\begin{proof}[Proof of Corollary \ref{cor:2,3}] The manifold $-\Sigma(2,3,5)$ is $+1$ surgery on the $(2,3)$ torus knot of genus 1. By Theorem \ref{thm:genus1} it remains to realize the listed lattices. The corresponding surgery cobordism provides the form $\langle +1 \rangle$, and $-\Sigma(2,3,5)$ bounds a plumbed manifold with lattice $E_8$. After connect summing with copies of $\cp$ we obtain from these $\langle +1 \rangle^{n+1}$ and $\langle +1 \rangle^n \oplus E_8$ for $n\geqslant 0$. Finally, $\langle -1 \rangle^n$ cannot occur; for if it did, gluing the orientation reversed 4-manifold to the $E_8$ plumbing would yield a non-diagonal definite lattice $E_8\oplus \langle + 1\rangle^n$, contradicting Donaldson's theorem.
\end{proof}

\vspace{0.25cm}

As mentioned in the introduction, Corollary \ref{cor:2,3} is a slight improvement on the main result of Fr\o yshov's PhD thesis \cite{froyshov-thesis}. Although the proof above used some Heegaard Floer theory, we will remove this dependency in the next section. In Section \ref{sec:alt} we provide another proof of Corollary \ref{cor:2,3} which is closer to Fr\o yshov's proof.

%% file: genus2.tex
%!TEX root = main.tex

In this section we prove Theorem \ref{thm:genus2}. We continue our notation of lattices from Section \ref{sec:genus1}. We begin with a family of examples for later reference. Using notation of \cite{froyshov-inequality} set
\begin{equation}
    \Gamma_{4k} \; = \; \left\{ (x_1,\ldots,x_{4k}) \in \Z^{4k}\cup \left(v+\Z^{4k}\right) : \; \sum x_i \equiv 0 \text{ (mod 2)}\right\}\label{eq:gamma}
\end{equation}
where $v=(\tfrac{1}{2},\ldots,\tfrac{1}{2})\in \R^{4k}$. We remark that $\Gamma_{4}$ is diagonalizable, and $\Gamma_{8}=E_8$. The lattice $\Gamma_{4k}$ is even precisely when $k$ is even.  We note that $\Gamma_{12}$ is the same as $D_{12}$ from Table \ref{tab:table1}, the latter notation indicating that the root lattice of $\Gamma_{12}$ is $\mathsf{D}_{12}$. The lattice $\Gamma_{4k}$ is isomorphic to the intersection form of the positive definite plumbing with boundary the orientation-reversed Brieskorn sphere $-\Sigma(2,2k-1,4k-3)$:
\begin{figure}[h!]
\begin{center}
\begin{tikzpicture}
	\draw (0,0) -- (5,0);
	\draw (6,0) -- (7,0);
	\draw (2,0) -- (2,1);
	\draw[fill=black] (0,0) circle(.1);
	\draw[fill=black] (1,0) circle(.1);
	\draw[fill=black] (2,0) circle(.1);
	\draw[fill=black] (2,1) circle(.1);
	\draw[fill=black] (3,0) circle(.1);
	\draw[fill=black] (4,0) circle(.1);
	\draw[fill=black] (5,0) circle(.1);
	\draw[fill=black] (6,0) circle(.1);
	\draw[fill=black] (7,0) circle(.1);
	\node at (0,-.5) {$k$};
	\node at (1,-.5) {$2$};
	\node at (2,-.5) {$2$};
	\node at (3,-.5) {$2$};
	\node at (2.5,1) {$2$};
	\node at (4,-.5) {$2$};
	\node at (5.5,0) {$\cdots$};
	\node at (7,-.5) {$2$};
\end{tikzpicture}
\end{center}
\caption{}\label{fig:gamma}
\end{figure}
Via (\ref{eq:gamma}), the node $k$ corresponds to the vector $( \frac{1}{2},\ldots,\frac{1}{2})$, while the other nodes correspond to $(1,1,0\ldots,0)$ and $(1,-1,0,\ldots,0), \ldots,(0,\ldots,1,-1,0)$. Replacing $(\frac{1}{2},\ldots,\frac{1}{2})$ by $(0,\ldots,0,1,-1)$ in this collection yields the root lattice $\mathsf{D}_{4k}\subset \Gamma_{4k}$.

\vspace{0.25cm}

\begin{prop}\label{prop:gamma}
$   m(\Gamma_{4g+4}) \geqslant g$ and $ N_\beta^4(g)\geqslant f_4(\Gamma_{4g+4})\geqslant \lceil g/2 \rceil$.
\end{prop}

\vspace{0.15cm}

\begin{proof} It is shown in \cite[\S 2]{fs-can} that $R(k)=\cp\# (4k+1)\overline{\cp}$ can be decomposed as $W\cup N$, where $W$ is the negative definite plumbing of $\Sigma(2,2k-1,4k-3)$ with intersection form $-\Gamma_{4k}$ and $N$ is obtained from attaching to the 4-ball a $0$-framed 2-handle along the $(2,2k-1)$ torus knot and a $(-1)$-framed 2-handle along a meridian of the torus knot. Blowing down the meridian 2-handle yields $X(k)$ such that $R(k)=X(k)\#\overline{\cp}$, with a decomposition $W\cup N'$ where $N'$ is obtained by attaching only a $(+1)$-framed 2-handle to the torus knot. Since the $(2,2k-1)$ torus knot has genus $k-1$, the 2-handle can be capped off to form a surface $\Sigma(k)\subset X(k)$ of genus $k-1$. The lattice of vectors vanishing on $[\Sigma(k)]$ is isomorphic to $-\Gamma_{4k}$.\\

The vector $w=(\frac{1}{2},\ldots,\frac{1}{2})\in \mathcal{L}:=\Gamma_{4g+4}$ is extremal with $w^2=g+1$ and $\text{Min}(w+2\mathcal{L})=\{w,-w\}$. Thus $ m(\mathcal{L}) \geqslant w^2-1 = g$. It also follows that $f_4(\mathcal{L})\geqslant \lceil m(\mathcal{L})/2 \rceil \geqslant \lceil g/2 \rceil$. Now given $g$ we take as our 4-manifold $X=X(g+1)$ with genus $g$ surface $\Sigma=\Sigma(g+1)$. Then the left side of (\ref{eq:ineqchar4}) is $N_\beta^4(g)$, and the result follows.
\end{proof}

\vspace{0.25cm}

This should be compared to \cite[Prop.1]{froyshov-inequality}. There it is shown that $e_0(\Gamma_{4g+4})= \lceil g/2 \rceil$. Thus the above family of 4-manifolds with surface achieve sharpness in Fr\o yshov's inequality of Theorem \ref{thm:char0}. Proposition \ref{prop:nilppartial} shows that the same family achieves sharpness in the inequality of Theorem \ref{thm:char4} for low $g$, and we expect this to be true for all $g$. If inequality \eqref{eq:mod2speculation} in the context of mod 2 coefficients were to hold in general, then this family would achieve sharpness there as well. We remark that the same 4-manifolds are used by Behrens and Golla \cite{bg} in the Heegaard Floer context.\\

We now move on to the main line of argument for Theorem \ref{thm:genus2}. Recall that for the proof of Theorem \ref{thm:genus1}, we used Lemma \ref{lemma:indecomp}, which says $f_4(\mathcal{L})=1$ implies the root lattice of $\mathcal{L}$ is indecomposable. The key algebraic input towards the proof of Theorem \ref{thm:genus2} is the following upgrade.

\vspace{0.25cm}

\begin{lemma}\label{lemma:f4}
    If $f_4(\mathcal{L})=1$ then $L$ is one of $E_8$ or $\Gamma_{12}$.
\end{lemma}

\vspace{0.15cm}

\begin{proof}
    From Lemma \ref{lemma:indecomp} we know $R$ is indecomposable, and hence one of $\A_n$, $\Ds_n$, $\E_6$, $\E_7$, $\E_8$ or zero. We will again use that $w\in R$ with $w^2=4$ has $\text{Min}(w+2\mathcal{L})=\text{Min}(w+2R)$, as shown in the proof of Lemma \ref{lemma:indecomp}. All extremal vectors $w$ chosen below have the property that the elements in $\text{Min}(w+2\mathcal{L})$ have the same signs in the expression for $\eta$ when $m=0$.\\
    
    Suppose $R=\mathsf{E}_7$. A standard model for $\mathsf{E}_7$ is the sublattice of $E_8=\Gamma_8$ consisting of vectors whose coordinates add to zero. Let $w=(1,1,-1,-1,0,0,0,0)$. Then $w$ is extremal in $\mathsf{E}_7$ of square 4, and $\text{Min}(w+2\mathcal{L})=\text{Min}(w+2\mathsf{E}_{7})$ consists of the 12 vectors obtained by permuting the signs of $w$ and those of $(0,0,0,0,1,1,-1,-1)$. Thus $\eta(\mathcal{L},w)=6 \not\equiv 0$ (mod 4), and $f_4(\mathcal{L})\geqslant w^2/2= 2$.\\
    
    Suppose $R=\mathsf{E}_6$. A standard model for $\mathsf{E}_6$ is the sublattice of $E_8=\Gamma_8$ consisting of vectors whose last three coordinates are equal. Consider $w=(1,1,1,1,0,0,0,0)\in\mathsf{E}_6$, extremal and of square 4. Then $\text{Min}(w+2\mathcal{L})=\text{Min}(w+2\mathsf{E}_{6})$ consists of the 8 vectors $(\pm 1,\pm 1,\pm 1,\pm 1,0,0,0,0)$ with an even number of signs, as well as the 2 vectors $\pm(0,0,0,0,1,1,1,1)$. Thus $\eta(\mathcal{L},w)=(8+2)/2=5 \not\equiv 0$ (mod 4), and $f_4(\mathcal{L})\geqslant w^2/2= 2$.\\
    
    Suppose $R=\mathsf{A}_n$, $n\geqslant 3$. As in the proof of Lemma \ref{lemma:e8}, take $w$ to be the vector given by $(1,1,-1,-1,0,\ldots,0)\in\mathsf{A}_n$, for which $w+2\mathsf{A}_n$ has 6 extremal vectors. Then $\eta(\mathcal{L},w)=3 \not\equiv 0 $ (mod 4), and $f_4(\mathcal{L})\geqslant w^2/2=2$.\\
    
    Suppose $R=\mathsf{A}_2$. Let $\pi:R\otimes \Z/2\to L\otimes \Z/2$ be the map induced by inclusion. This map cannot be onto, since any unimodular lattice of rank $2$ is diagonal. Choose $w\in L$ of minimal norm such that $[w]\not\in\text{im}(\pi)$. We showed in the proof of Lemma \ref{lemma:e8} that $\text{Min}(w+2\mathcal{L})=\{w,-w\}$. Since $w\not \in R$, $w^2\geqslant 3$. If $w^2\geqslant 4$ then $f_4(\mathcal{L})\geqslant \lfloor w^2/2 \rfloor \geqslant 2$. So suppose $w^2=3$. Further suppose $w\perp R$. Then $w + r$ is extremal of square 5 and $\text{Min}(w+r+2\mathcal{L})=\{\pm w \pm r\}$. We compute $\eta(\mathcal{L},w+r,w,1) =-2w^2=-6 \not \equiv 0 \text{ (mod 4)}$. It follows that $f_4(\mathcal{L})\geqslant((w+r)^2-1)/2 =2$. Now instead suppose $w$ is not orthogonal to $R$. From $5\pm 2w\cdot r = (w\pm r)^2 \geqslant 0$ and the assumption that $L$ has no vectors of square $1$ we obtain $|w\cdot r |\leqslant 1$ for each root $r$. Let $r_1,r_2,r_3$ be roots satisfying $r_1+r_2+r_3=0$, so that $\{\pm r_1, \pm r_2, \pm r_3\}$ is the set of all roots. The condition $|w\cdot r|\leqslant 1$ implies, after possibly relabeling, that $w\cdot r_1 = 0$, $w\cdot r_2 = 1$ and $w\cdot r_3 = -1$. Then $w+r_1$ is an extremal vector of square 5, $\text{Min}(w+r_1+2\mathcal{L})=\{\pm w \pm r_1, \pm(w+r_1+2r_3)\}$, and $\eta(\mathcal{L},w+r_1,w,1) =-7 \not \equiv 0 \text{ (mod 4)}$, again implying $f_4(\mathcal{L})\geqslant 2$.\\
    
    Suppose $R=\mathsf{A}_1$. Again, $\pi$ is not onto, and its cokernel has rank at least 2, since no unimodular lattice of rank $\leqslant 3$ has root lattice $\mathsf{A}_1$. Again choose $w$ of minimal norm such that $[w]\not\in\text{im}(\pi)$. If $w^2\geqslant 4$, we are done; so suppose $w^2=3$. Let $r$ be the unique root in $\mathsf{A}_1$ up to sign. If $w\cdot r=0$, then as in the case for $\mathsf{A}_2$ we can use $w+r$ to conclude $f_4(\mathcal{L})\geqslant 2$. So assume $w\cdot r \neq 0$. As before, $|w\cdot r|\leqslant 1$, so in fact $|w\cdot r|=  1$. Let $v$ be of minimal norm such that $[v]\not\in \text{im}(\pi)+[w]$. Then the same argument as in the proof of Lemma \ref{lemma:e8} shows $\text{Min}(v+2\mathcal{L})=\{v,-v\}$. If $v^2\geqslant 4$, we are done; so suppose $v^2=3$. If $v\cdot r=0$, then take $v+r$ as in the case of $\mathsf{A}_2$. Suppose instead $v\cdot r\neq 0$. As with $w$, we have $|v\cdot r|\leqslant 1$, so $|v\cdot r|=1$. Since $[v\pm w]\not\in \text{im}(\pi)+[w]$, by minimality of $v$ we have $(w\pm v)^2 \geqslant v^2$, from which it follows that $|w\cdot v|\leqslant 1$. If $w\cdot v=0$, then for some choice of signs, $w\pm r\pm v$ has square 4; if $w\cdot v = \pm 1$, then one of $v \pm w$ has square 4. In either case we obtain a vector of square 4, and take this as our extremal vector to obtain $f_4(\mathcal{L})\geqslant 2$.\\
    
    Next, suppose $R=\mathsf{D}_n$ for some $n\geqslant 4$. Suppose $\mathsf{D}_n$ has full rank within $L$, i.e. the map $\iota:\mathsf{D}_n\otimes \R \to L\otimes \R$ induced by inclusion is an isomorphism. The only full rank embeddings of $\mathsf{D}_{n}$ into a non-diagonal unimodular lattice $L$ are those inside $\Gamma_{4n}$ with $n\geqslant 2$ (see e.g. \cite[Sec.1.4]{ebeling}), and we have computed $f_4(\Gamma_{4n})\geqslant\lfloor n/2\rfloor$. If $f_4(\mathcal{L})=1$ then either $n=2$, in which case $L=E_8$, contradicting the assumption that $R=\mathsf{D}_8$, or $n=3$, in which case $L=\Gamma_{12}$. Thus we may assume that $\iota$ is not onto. It follows also that $\pi$ is not onto, since $n=\text{rank} (R) < \text{rank} (L)$. We will see that the arguments below generalize those for the cases of $\mathsf{A}_1$ and $\mathsf{A}_2$ given above.\\
    
    We begin as in the case for $\mathsf{A}_2$. Let $w\in L$ be of minimal norm such that $[w]\not\in \text{im}(\pi)$. If $w^2\geqslant 4$ we are done, as argued in the above cases, and so we may assume $w^2=3$. We may also assume $w\not\in\text{im}(\iota)$. Indeed, consider the map $L\to (L/\mathsf{D}_n)/\text{Tor}$. The codomain here is a free abelian group of rank equal to $\text{rank}(L)-n>0$. The argument in Lemma \ref{lemma:e8} shows that for a given proper subspace $S\subset L\otimes\Z/2$, any $w\in L$ of minimal norm among vectors such that $[w]\not\in S$ has $\text{Min}(w+2\mathcal{L})= \{w,-w\}$; in Lemma \ref{lemma:e8}, $S=\text{im}(\pi)$. In particular, we may choose $S$ to be the kernel of $p:L\otimes \Z/2\to (L/\mathsf{D}_n)/\text{Tor}\otimes\Z/2$. By construction, $w\not\in\text{im}(\iota)$.\\
    
    Choose a root $r\in L$ such that $w\cdot r=0$, following the argument as in the case of $\mathsf{A}_2$. Then $w+r$ is extremal of square $5$. Let $v\in \text{Min}(w+r+2\mathcal{L})$. Assume $v\neq \pm (w+r)$ and $v\cdot (w+r) \geqslant 0$. Write $v-w-r=2u$ where $u\in L$. Then $0\neq 4u^2 = (v-w-r)^2 \leqslant (w+r)^2 + v^2 =10$ implies $u$ is a root. Recall for any root $u$ that from the condition $(w\pm u)^2\geqslant 0$ we have $|w\cdot u| \leqslant 1$, and $|r\cdot u|\leqslant 1$ if $u\neq \pm r$ holds. Then we have
    \[
        5 \; = \; v^2 \; = \; (w+r+2u)^2 \; = \; 13 + 4(w\cdot u + r\cdot u)
    \]
implies either $w\cdot u = r\cdot u = -1$ or $u=-r$. Let $N$ be the set of roots $u$ such that $w\cdot u = r\cdot u = -1$. We conclude $\text{Min}(w+r+2\mathcal{L}) = \{\pm w \pm r\} \cup \{\pm (w+r+2u): u\in N\}$. Let $a\in \mathcal{L}$. We compute
    \begin{equation}
        \eta(\mathcal{L},w+r,a,1) \; = \; -(2 + |N|)w\cdot a - |N|r\cdot a - 2\sum_{u\in N}a\cdot u.\label{eq:etadn}
    \end{equation}
If we set $a=w$, using $w^2=3$, $w\cdot r = 0$ and the definition of $N$, from (\ref{eq:etadn}) we compute
    \begin{equation}
        \eta(\mathcal{L},w+r,w,1) \; = \; -6-|N|.\label{eq:etadnw}
    \end{equation}
If (\ref{eq:etadnw}) is nonzero modulo 4,  then $f_4(\mathcal{L})\geqslant ((w+r)^2-1)/2 = 2$ and we are done. So henceforth assume $|N|\equiv 2$ (mod 4).\\
    
    We represent $\mathsf{D}_n$ as the sublattice of $\mathbb{Z}^n$ of vectors whose coordinates sum to zero modulo $2$. Henceforth we identify the vectors in this representation of $\mathsf{D}_n$ with those in the root lattice of $L$. We may suppose that $r=(1,1,0,\ldots,0)=e_1+e_2$, since the automorphism group of $\mathsf{D}_n$ acts transitively on roots. Here we write $e_1,\ldots, e_n$ for the standard basis vectors of $\mathbb{Z}^n$. Then the vectors
    \begin{equation}
        r_{h,i}^\pm \; := \; -e_h \pm  e_i, \qquad h\in \{1,2\}, \;\; 3\leqslant i \leqslant n\label{eq:rij}
    \end{equation}
make up the set of roots $u$ such that $r\cdot u=-1$. For a fixed $i$ we have the two relations
    \begin{align}
        r + r_{1,i}^+ + r_{2,i}^- \; & = \; 0\label{eq:rootrel1}\\
        r + r_{1,i}^- + r_{2,i}^+ \; & = \; 0\label{eq:rootrel2}
    \end{align}
    Pairing (\ref{eq:rootrel1}) with $w$, we see that either $w\cdot r_{1,i}^+=w\cdot r_{2,i}^-=0$, or $w\cdot r_{1,i}^+=-w\cdot r_{2,i}^-=\pm 1$. Similarly for (\ref{eq:rootrel2}). Thus $N_i  := N\cap \{r_{1,i}^+,r_{1,i}^-,r_{2,i}^+,r_{2,i}^-\}$ has $0$, $1$ or $2$ elements. Furthermore, $N=\cup_{i=3}^n N_i$.\\

    Now let $I\subset\{3,\ldots,n\}$ with $|I|$ even. Then there exists $a\in \mathcal{L}$ such that
    \begin{equation}
        \eta(\mathcal{L},w+r,a,1) \; \equiv \; 2\sum_{i\in I} |N_i| \;\;\;\; (\text{mod  4}).\label{eq:nnn}
    \end{equation}
To see this, let $a$ be the vector corresponding to $(a_1,\ldots,a_n)\in\mathsf{D}_n$ which has $a_i=1$ if $i\in I$ and $a_i=0$ otherwise, and then compute (\ref{eq:nnn}) using (\ref{eq:etadn}). From (\ref{eq:nnn}) we may assume that either (I) $|N_i|=1$ for all $i$ or (II) $|N_i|\in\{0,2\}$ for all $i$. Indeed, if $|N_{j}|=1$ and $|N_k|\in\{0,2\}$ for some $j,k$ then setting $I=\{j,k\}$ in (\ref{eq:nnn}) yields $\eta(\mathcal{L},w+r,a,1)\equiv 2 \not \equiv 0$ (mod 4).\\

    Case (I). Suppose $|N_i|=1$ for all $3\leqslant  i \leqslant n$. Then $|N|=\sum |N_i| = n-2$. Having assumed $|N|\equiv 2$ (mod 4), we conclude $n\equiv 0$ (mod 4). Set $r_1:=e_1-e_2\in \mathsf{D}_n$. Since $\smash{r_1 + r_{1,i}^+ = r_{2,i}^+}$, and $|N_i|=1$ implies one of $\smash{r_{1,i}^+}$ or $\smash{r_{2,i}^+}$ is orthogonal to $w$ and the other has inner product $\pm 1$ with $w$, we obtain $|w\cdot r_1|=1$. In a similar fashion, for each $3\leqslant i\leqslant n$ let $\smash{s_i,t_i\in\{r_{1,i}^+,r_{1,i}^-\}}$ be such that $|w\cdot s_i|=1$ and $w\cdot t_i=0$. For $2\leqslant i \leqslant n/2$ set $r_i:=s_{2i-1}-t_{2i}$. This is a vector in $\mathsf{D}_n$ whose $(2i-1)^\text{th}$ and $2i^\text{th}$ entries are $\pm 1$, with all other coordinates zero. Then $r_1,\ldots,r_{n/2}$ are orthogonal roots all satisfying $|w\cdot r_i|=1$. Since $w\not\in \text{im}(\iota)$, its length is strictly greater than that of its projection onto the span of the subspace in $\mathsf{D}_n\otimes\R$ generated by the $r_i$:
    \begin{equation}
        3 \; = \; w^2  \; > \; \sum_{i=1}^{n/2} \frac{|w\cdot r_i|^2}{r_i^2} \; = \; n/4.\label{eq:proj}
    \end{equation}
Recalling $n\geqslant 4$ and $n\equiv 0$ (mod 4), we must have $n\in\{4,8\}$, i.e. $R\in\{\mathsf{D}_4,\mathsf{D}_8\}$. Before considering these two cases separately, we determine one more constraint. Suppose $\smash{N_j=\{r_{1,j}^+\}}$ and $\smash{N_k = \{r_{2,k}^+\}}$ for some $j\neq k$; the superscripts here are not important. Then $\smash{u=r_{1,j}^+-r_{1,k}^-}$ is a root for which $u\cdot w = -2$, a contradiction. Thus $\smash{N_i = \{r_{h,i}^{\sigma_i}\}}$ for each $i$, for some uniform $h\in\{1,2\}$, and each $\sigma_i\in\{\pm\}$. We conclude that after perhaps reflecting some coordinates in the range $3\leqslant i\leqslant n$ and permuting the first two coordinates in our representation of $\mathsf{D}_n$ we have $\smash{N=\{r_{1,3}^+,\ldots,r_{1,n}^+\}}$.\\
    
    Now suppose $R=\mathsf{D}_4$. Setting $w_1=w$, we choose $w_2,\ldots,w_k$ of minimal norm such that $[w_i]\not\in \text{ker}(p)+\sum_{j<i} [w_j]$. We may suppose each $w_i^2=3$, or else we are done. Our previous arguments show $|w_i\cdot w_j|\leqslant 1$ for $i\neq j$ and $|w_i\cdot u|\leqslant 1$ for all roots $u$. We only need to do this for $k=3$, which is possible because there are no definite unimodular lattices of rank $<4+3$ with root lattice $\mathsf{D}_4$; the first non-diagonal definite unimodular lattice, by rank, is $E_8$. By our assumption from the previous paragraph, $N=\{r_{1,3}^+,r_{1,4}^+\}$. Define the dual lattice of $\mathsf{D}_n$ to be $\mathsf{D}_n^\ast=\{x\in\mathsf{D}_n\otimes\R: \; x\cdot y \in \Z,\; \forall y\in\mathsf{D}_n\}$, and let
    \[
        L \longrightarrow \mathsf{D}_n^\ast, \qquad w\longmapsto \overline{w}
    \]
denote projection. The values $w\cdot u$ for all roots $u\in \mathsf{D}_4$ are determined and given by column (i) in Table \ref{table:d4}, which lists one root for each pair $\{u,-u\}\in \text{Roots}(\mathsf{D}_4)/\pm$. In particular, we see that $\overline{w}=(\frac{1}{2},-\frac{1}{2},-\frac{1}{2},-\frac{1}{2})\in\mathsf{D}_4^\ast$. Note $w$ is orthogonal to exactly half the roots in $\mathsf{D}_4$. We may also assume case (I) for $w_2$ and $w_3$, each with respect to some orthogonal root. (If either is case (II), move to case (II).) Then, just as was established for $w$, each of $w_2,w_3$ is orthogonal to half the roots of $\mathsf{D}_4$. Thus two of $w_1,w_2,w_3$ are orthogonal to a common root. Without loss of generality, suppose these two vectors are $w=w_1$ and $v\in\{w_2,w_3\}$, and that the orthogonal root is $r$. Recalling $|N|\equiv 2$ (mod 4), formula (\ref{eq:etadn}) yields
    \begin{equation}
        \eta(\mathcal{L},w+r,v,1) \; \equiv \; 2(r_{1,3}^+ + r_{1,4}^+)\cdot v \;\; (\text{mod 4}).\label{eq:etarootorth}
    \end{equation}
    Thus we may assume that $v$ is either orthogonal to $N$ or pairs non-trivially to $\pm 1$ with both of its vectors. Combining this with the constraints for $v$ previously determined for $w$, the pairings of $v$ with the roots of $\mathsf{D}_4$ must be given by one of columns (i)-(iv) in Table (\ref{table:d4}). In particular, $\overline{v}=\pm \overline{w}$ or $\overline{v}=\pm (\frac{1}{2},-\frac{1}{2},\frac{1}{2},\frac{1}{2})$. The case of $\mathsf{D}_4$ will now be completed by constructing an extremal vector $x$ of square 4 such that $\eta(\mathcal{L},x)\not\equiv 0$ (mod 4), following the cases of $\mathsf{A}_1$ and $\mathsf{A}_2$ above. There are two cases to consider: $w\cdot v=0$ and $w\cdot v=\pm 1$.\\
    
\begin{figure}[t]
\centering{
\begin{tabular}{lrrrr}
$\text{Roots}(\mathsf{D}_4)/\pm$  & (i) & (ii)  & (iii) & (iv) \\
\hline
$(\phantom{-}1,\phantom{-}1,\phantom{-}0,\phantom{-}0)$ & $0$ & $0$ & $0$ & $0$ \\
$(\phantom{-}1,-1,\phantom{-}0,\phantom{-}0)$ & $1$ & $-1$ & $1$ & $-1$ \\
$(-1,\phantom{-}0,\phantom{-}1,\phantom{-}0)$  & $-1$ & $1$ & $0$ & $0$ \\
$(-1,\phantom{-}0,-1,\phantom{-}0)$ & $0$ &  $0$ & $-1$ & $1$ \\
$(-1,\phantom{-}0,\phantom{-}0,\phantom{-}1)$ & $-1$ & $1$ & $0$ & $0$ \\
$(-1,\phantom{-}0,\phantom{-}0,-1)$ & $0$ & $0$ & $-1$ & $1$ \\
$(\phantom{-}0,-1,-1,\phantom{-}0)$  & $1$ & $-1$ & $0$  & $0$ \\
$(\phantom{-}0,-1,\phantom{-}1,\phantom{-}0)$  & $0$ & $0$ & $1$ & $-1$ \\
$(\phantom{-}0,-1,\phantom{-}0,-1)$ & $1$ & $-1$ & $0$ & $0$ \\
$(\phantom{-}0,-1,\phantom{-}0,\phantom{-}1)$ & $0$ & $0$ & $1$ & $-1$ \\
$(\phantom{-}0,\phantom{-}0,\phantom{-}1,\phantom{-}1)$ & $-1$ & $1$ & $1$ & $-1$ \\
$(\phantom{-}0,\phantom{-}0,\phantom{-}1,-1)$ & $0$ & $0$ & $0$ & $0$ \\
\end{tabular}
}
\captionof{table}[]{}\label{table:d4}
\end{figure}
    
    First, suppose $w\cdot v =0$. Upon possibly replacing $v$ with $-v$, the pairings of $v$ with $\mathsf{D}_4$ are given by either (i) or (iii) in Table \ref{table:d4}. Then $x=w+v+s$ is of square $4$, where $s=-r_1=e_2-e_1$. As usual, if $x+2u$ is extremal, then $u$ is a root, and $(x+2u)^2=4$ implies $u\cdot x=-2$, and thus
    \begin{equation}
        \text{Min}(x+2\mathcal{L}) \; = \; \{\pm x\}\cup\{\pm(x+2u): \; u\cdot (w+v+s)=-2,\; u^2=2, \; u\in L\}.\label{eq:mind41}
    \end{equation}
    Now each of $y\in \{w,v,s\}$ has $|y\cdot u|\leqslant 1$ for any root $u\neq \pm r_1$, and so $u$ in (\ref{eq:mind41}) must be orthogonal to one of the three, and have pairing $-1$ with the other two. If $v$ has pairings given by (i) in Table \ref{table:d4}, then the set of such $u$ is given by $\{e_3+e_4\}$. Thus $\text{Min}(x+2\mathcal{L})/\pm$ has 2 elements, implying $f_4(\mathcal{L})\geqslant 2$. If instead $v$ corresponds to column (iii) in Table (\ref{table:d4}), then there are no solutions to $u\cdot (w+v+s)=-2$, and $\text{Min}(x+2\mathcal{L})/\pm$ has 1 element, again implying $f_4(\mathcal{L})\geqslant 2$.\\
    
    Next, suppose $|w\cdot v|=1$. Upon possibly replacing $v$ with $-v$ we may suppose $w\cdot v = -1$. Then $x=w+v$ is extremal of square 4. Further, if $w+v+2u$ is extremal, then $(w+v+2u)^2=4$ implies $u\cdot (w+v)=-2$. Since for $y\in\{w,v\}$ and any root $u$ we have $|y\cdot u|\leqslant 1$, it follows that
    \begin{equation}
        \text{Min}(x+2\mathcal{L}) \; = \; \{\pm x\}\cup\{\pm(x+2u): \; u\cdot w = u\cdot v=-1,\; u^2=2,\; u\in L\}.\label{eq:mind42}
    \end{equation}
    If $v$ has pairings with $\mathsf{D}_4$ the same as that of $w$, in (i) of Table \ref{table:d4}, then there are 6 such roots $u$; for (ii) there are zero; and for (iii) and (iv) there is 1. Thus $\text{Min}(x+2\mathcal{L})/\pm$ has either 7, 1 or 2 elements, all nonzero modulo 4. Thus $f_4(\mathcal{L})\geqslant 2$. This completes the case of $\mathsf{D}_4$ within case (I).\\
    
    Now suppose $R=\mathsf{D}_8$. In this case we have assumed $\smash{N=\{r_{1,3}^+,r_{1,4}^+,r_{1,5}^+,r_{1,6}^+,r_{1,7}^+,r_{1,8}^+\}}$. This implies in particular that $\overline{w}=(\frac{1}{2},-\frac{1}{2},-\frac{1}{2},-\frac{1}{2},-\frac{1}{2},-\frac{1}{2},-\frac{1}{2},-\frac{1}{2})\in \mathsf{D}_8^\ast$. As in the case of $\mathsf{D}_4$ we can find $v\in L$ of minimal squared norm, which we may assume is $3$, such that $[v]\not\in \text{ker}(p)+[w]$ and $v\cdot r=0$. Indeed, to adapt the above argument, where $v\in\{w_2,w_3\}$, we only need to note that there are no unimodular definite lattices of rank $<8+3$ with root system $\mathsf{D}_8$; this is well-known, and is verified, for example, by \cite[Table 16.7]{conwaysloane}. The analogue of (\ref{eq:etarootorth}) here is
    \begin{equation}
        \eta(\mathcal{L},w+r,v,1) \; \equiv 2\sum_{i=3}^8 r_{1,i}^+\cdot v \;\; (\text{mod 4})\label{eq:etanearend}
    \end{equation}
    The constraint that (\ref{eq:etanearend}) is zero modulo 4, along with the constraints previously determined for $w$, imply, after possibly an automorphism of our representation of $\mathsf{D}_8$ permuting and reflecting coordinates, that $\overline{w}=v_1$ and $\overline{v}\in\{\pm v_1,\pm v_2,\pm v_3,\pm v_4\}$ where
    \begin{align*}
        v_1 & =  (+\tfrac{1}{2},-\tfrac{1}{2},-\tfrac{1}{2},-\tfrac{1}{2},-\tfrac{1}{2},-\tfrac{1}{2},-\tfrac{1}{2},-\tfrac{1}{2})\\
        v_2 & =  (+\tfrac{1}{2},-\tfrac{1}{2},-\tfrac{1}{2},-\tfrac{1}{2},-\tfrac{1}{2},-\tfrac{1}{2},+\tfrac{1}{2},+\tfrac{1}{2})\\
        v_3 & =  (+\tfrac{1}{2},-\tfrac{1}{2},-\tfrac{1}{2},-\tfrac{1}{2},+\tfrac{1}{2},+\tfrac{1}{2},+\tfrac{1}{2},+\tfrac{1}{2})\\
        v_4 & =  (+\tfrac{1}{2},-\tfrac{1}{2},+\tfrac{1}{2},+\tfrac{1}{2},+\tfrac{1}{2},+\tfrac{1}{2},+\tfrac{1}{2},+\tfrac{1}{2})
    \end{align*}
    \noindent Observe that $w-\overline{w}\in L\otimes \R$ has square 1, and is orthogonal to $\text{im}(\iota)$. Projecting $v$ onto the subspace spanned by $\text{im}(\iota)$ and $w-\overline{w}$ we obtain the following:
    \begin{equation}
        3 \; = \; v^2 \; > \; \left(v\cdot (w-\overline{w})\right)^2 + \overline{v}^2 \; = \; \left(v\cdot w-\overline{v}\cdot \overline{w}\right)^2 + 2.\label{eq:lastineqihope}
    \end{equation}
    First suppose $w\cdot v=0$. Then (\ref{eq:lastineqihope}) implies $\overline{w}\cdot \overline{v}=0$. We must have $\overline{v}=\pm v_3$. Upon possibly replacing $v$ by $-v$ we may assume $\overline{v}=v_3$. Then $x=w+v+s$ with $s=-r_1$ is extremal of square 4. The only root $u\in\mathsf{D}_8$ satisfying $u\cdot(w+v+s)=-2$ is $e_3+e_4$, and by (\ref{eq:mind41}) we have that $\text{Min}(x+2\mathcal{L})/\pm$ is of cardinality 2, implying $\eta(\mathcal{L},x)\equiv 2 \not\equiv 0$ (mod 4) and $f_4(\mathcal{L})\geqslant 2$.\\
    
    Now suppose $|w\cdot v|=1$. Upon possibly replacing $v$ with $-v$ we may assume $w\cdot v =-1$. Then (\ref{eq:lastineqihope}) implies $\overline{v}=-v_2$ or $\overline{v}=v_4$. When $\overline{w}=v_1$ and $\overline{v}=-v_2$ the only root $u\in\mathsf{D}_8$ satisfying $u\cdot w = u\cdot v = -1$ is $e_7+e_8$. When $\overline{w}=v_1$ and $\overline{v}=v_4$, the only such root is $-e_1+e_2$. In either case, (\ref{eq:mind42}) implies $\text{Min}(w+v+2\mathcal{L})/\pm$ has 2 elements, and thus $\eta(\mathcal{L},w+v)\equiv 2 \not\equiv 0$ (mod 4) and $f_4(\mathcal{L})\geqslant 2$. This completes the case of $\mathsf{D}_8$ within case (I), and of case (I) entirely.\\
    
    Case (II). Suppose $|N_i|\in\{0,2\}$ for $3\leqslant i \leqslant n$. Let $I_w\subset\{3,\ldots,n\}$ be the set of $i$ such that $|N_i|=2$. Since $|N|= 2|I_w|\equiv  2$ (mod 4), $I_w$ is nonempty. Recall $r_1=e_1-e_2$. Let $i\in I_w$. As $w$ pairs non-trivially with all of $\smash{r_{1,i}^+,r_{1,i}^-,r_{2,i}^+,r_{2,i}^-}$, and $\smash{r_1 + r_{1,i}^+ = r_{2,i}^+}$, we have $w\cdot r_1=0$. Then $w\cdot r=w\cdot r_1=0$ implies $w\cdot e_1=w\cdot e_2=0$, the latter computation holding in $L\otimes \R$. As $e_i = e_1+r_{1,i}^+$, we have $|w\cdot e_i|=1$ for $i\in I_w$ within $L\otimes\R$. Because $w\not\in\text{im}(\iota)$ we have
    \begin{equation}
        3 \; = \; w^2 \; > \; \overline{w}^2 \;= \; \sum_{i\in I_w} |w\cdot e_i|^2 \; = \; |I_w|.\label{eq:almostthere}
    \end{equation}
    With the constraint that $|I_w|$ is odd, this implies $|I_w|=1$. Without loss of generality we may assume $N=N_3$. After an automorphism of our representation of $\mathsf{D}_n$ we may assume $N=\{r_{1,3}^+,r_{2,3}^+\}$. In particular, $\overline{w}=-e_3=(0,0,-1,0,\ldots,0)\in\mathsf{D}_n^\ast$.\\
    
    Next, we claim $L\otimes\Z/2\neq \text{im}(\pi)+[w]$. Suppose to the contrary equality holds here. Then, since by assumption $\text{rank}(L)>n$, as follows from $\iota$ having kernel, we must have $\text{rank}(L)=n+1$, and that $\pi$ is injective. Note, however, that $[2e_1]\neq 0 \in \mathsf{D}_n\otimes\Z/2$ and $\pi([2e_1])$ pairs trivially with $\text{im}(\pi)+[w]$, contradicting the non-degeneracy of the pairing on $L\otimes \Z/2$, the latter of which follows from the unimodularity of $L$. This verifies the claim. Thus we may choose a vector $v$ of minimal norm such that $[v]\not\in\text{im}(\pi)+[w]$, which, as usual, we may suppose has $v^2=3$. We assume $v$ has the type of case (II) as well; otherwise move to the paragraph following case (II).\\
    
    Now $v$ satisfies (\ref{eq:almostthere}) with the provision that strict inequality may not hold, and with $|I_v|$ on the right side defined using some root orthogonal to $v$ in place of $r$. The inequality is not necessarily strict because we have not claimed $v\not\in\text{im}(\iota)$. We conclude $|I_v|\in \{1,3\}$. If $|I_v|=3$, then $v=\overline{v}$. Furthermore, after an automorphism of $\mathsf{D}_n$, we may suppose $\overline{v}=e_i-e_j-e_k$ for some distinct $i,j,k$, similar to the determination $\overline{w}=-e_3$ above. But then $v+e_j+e_k = \overline{v}+e_j+e_k$ is a vector of square 1 in $L$, a contradiction. Thus we may assume $|I_v|=1$.\\

    Let $R_w$ be the number of roots orthogonal to $w$. Then $\overline{w}=-e_3$ implies $R_w=2(n-1)(n-2)$. Similarly, since $|I_v|=1$, we have $R_v=R_w$. If $n\geqslant 5$, then $R_w+R_v = 4(n-1)(n-2)> 2n(n-1)$, the total number of roots in $\mathsf{D}_n$, so that $w$ and $v$ must share a common orthogonal root. If $n=4$ we may argue as in case (I), using that there are no unimodular lattices with root lattice $\mathsf{D}_n$ of rank less than $3+4$, to sequentially choose $w_2,w_3$ and then choose $v\in\{w_2,w_3\}$. Thus without loss of generality, $v$ and $w$ are both orthogonal to a common root, which we may suppose is $r$.\\
    
    It follows then that $\overline{v}=\pm e_i$ for some $3 \leqslant i \leqslant n$. First suppose $i\neq 3$. Without loss of generality we may assume $\overline{v}=-e_4$. Our minimality assumption on $v$ implies $|w\cdot v|\leqslant 1$. Suppose $w\cdot v=0$. Consider the extremal vector $x=w+v+s$ of squared norm 4, where $s=e_3+e_4$. There are no roots $u$ satisfying $u\cdot (w+v+s)=-2$, and so (\ref{eq:mind41}) implies $\text{Min}(x+2\mathcal{L})/\pm$ has 1 element, whence $f_4(\mathcal{L})\geqslant x^2/2= 2$. If instead $|w\cdot v|=1$, then consider $x=w\pm v$, with the sign chosen so that $x$ is extremal of square 4. There is only one root $u$ such that $u\cdot w= u\cdot v=-1$, and so (\ref{eq:mind42}) implies $\text{Min}(x+2\mathcal{L})/\pm$ has 2 elements, whence $f_4(\mathcal{L})\geqslant 2$.\\
    
    Now suppose $\overline{v}=\pm e_3$. Consider the case $w\cdot v=0$. Upon perhaps replacing $v$ by $-v$ we may assume $\overline{v}=-e_3$. Then $x=w+v+s$, with $s$ as before, is extremal of square 4. The only root $u$ satisfying $u\cdot (w+v+s)=-2$ is $e_3-e_4$, so (\ref{eq:mind41}) implies $\text{Min}(x+2\mathcal{L})/\pm$ has 2 elements, whence $f_4(\mathcal{L})\geqslant x^2/2= 2$. Now consider $\overline{v}=\pm e_3$ and $|w\cdot v|=1$. Upon perhaps replacing $v$ with $-v$ we may suppose $w\cdot v=-1$. Then $x=w+v$ is an extremal vector of square 4. The roots $u$ satisfying $u\cdot v = u\cdot w=-1$ are (i) none, if $\overline{v}=+e_3$, or (ii) $e_3\pm e_i$ for $i\neq3$, if $\overline{v}=-e_3$, of which there are $2(n-1)$ many. Then (\ref{eq:mind42}) implies $\text{Min}(x+2\mathcal{L})/\pm$ has either (i) 1 element or (ii) $1+ 2(n-1)$ elements, both of which are odd numbers, and hence imply $f_4(\mathcal{L})\geqslant 2$. This completes case (II).\\

In the above treatment we assumed the vectors of square $3$ used were all in either case (I) or case (II). Suppose we encounter at least one of each type. Then by the argument in case (I), $n\in \{4,8\}$. As there are no unimodular definite lattices with root lattice $\mathsf{D}_4$ or $\mathsf{D}_8$ of rank $<16$ other than $E_8$ and $\Gamma_{12}$, we can sequentially choose square $3$ extremal vectors so that we have $3$ such vectors in either case (I) or case (II), which is enough to make either of the above arguments go through. This completes the case $\mathsf{D}_n$ ($n\geqslant 4$) entirely.\\
    
    Finally, suppose $L$ has no roots. Let $w\in L$ be of minimal nonzero norm. Then $w^2\geqslant 3$ and by the usual argument $\text{Min}(w+2\mathcal{L})=\{w,-w\}$. If $w^2\geqslant 4$ then $f_4(\mathcal{L})\geqslant \lceil f_2(\mathcal{L})/2\rceil \geqslant 2$. So suppose $w^2=3$. Let $v\in L$ be of minimal norm such that $[v]\not \in \{0, [w]\}\subset L\otimes \Z/2$. Then $v$ is extremal and $\text{Min}(v+2L)=\{v,-v\}$ as in the proof of Lemma \ref{lemma:e8}. If $v^2\geqslant 4$ we are done. So suppose $v^2=3$.\\
    
    For $s,t\in L$ of square 3 and $s\cdot t\leqslant 0$ we have $(s +t)^2=6 + 2s\cdot t\leqslant  6 $. Because $L$ has no vectors of square 2, the vector $s+t$ has square $4$ or $6$. In the former case, $\text{Min}(s+t+2\mathcal{L})=\{\pm(s+t)\}$ and $f_4(\mathcal{L})\geqslant 2$. So we may assume $(s+t)^2=6$, or equivalently $s\cdot t=0$. In particular, we may assume that any two vectors $s,t\in L$ of square 3 with $s\neq \pm t$ are orthogonal.\\
    
    Now consider $x=w+v$. This is extremal of square $6$. If $z=x+2u\in \text{Min}(x+2\mathcal{L})$ and $z\neq \pm x$, $z\cdot x\geqslant 0$, then $0\neq 4u^2 = (x-z)^2 \leqslant 12$ implies $u^2=3$, and $6=(x+2u)^2$ implies $u\cdot w+u\cdot v=-3$. By the assumption made at the end of the previous paragraph, we must have $u=-w$ or $u=-v$. Thus $\text{Min}(x+2\mathcal{L})=\{\pm w\pm v\}$. We then compute
    \begin{align*}
        \eta(\mathcal{L},x,wv,2) \; &= \; (-1)^{\left(\frac{x+w+v}{2}\right)^2}(w\cdot(w+v))(v\cdot(w+v)) \\
	& \qquad  + (-1)^{\left(\frac{x+w-v}{2}\right)^2}(w\cdot(w-v))(v\cdot(w-v)) \\
	\; & = \; 9+9 = 18 \not\equiv 0 \;\; \text{(mod 4)}.
    \end{align*}
    Thus $f_4(\mathcal{L})\geqslant (x^2-2)/2=2$. This completes the case of $L$ having no roots, and, having completed all cases, concludes the proof of the lemma.
    \end{proof}

    \vspace{0.25cm}

We obtain the following, which implies Theorem \ref{thm:genus2}.

\vspace{0.25cm}

\begin{corollary}\label{cor:genus2}
Suppose $Y$ is an integer homology 3-sphere which is $(+1)$ surgery on a knot $K$ in an integer homology 3-sphere with $g_{4,2}(K)=2$. If $X$ is a smooth, compact, oriented and definite 4-manifold bounded by $Y$ with non-diagonal lattice $\mathcal{L}$ and no 2-torsion, then the reduced part of $\mathcal{L}$ is isomorphic to either $E_8$ or $\Gamma_{12}$.
\end{corollary}

\vspace{0.15cm}

\begin{proof}
    Corollary \ref{cor:link} implies $f_4(\mathcal{L})\leqslant 1$. Since $\mathcal{L}$ is not diagonal, $f_4(\mathcal{L})=1$. By Lemma \ref{lemma:f4}, the reduced part of $\mathcal{L}$ must be one of $E_8$ or $\Gamma_{12}$.
\end{proof}

\vspace{0.15cm}

\begin{figure}[t]
\centering
\includegraphics[scale=1]{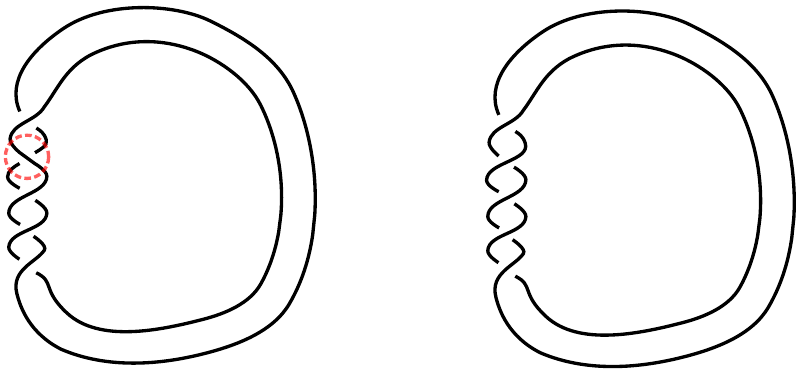}
\caption{{\small{The (2,3) torus knot, depicted on the left, is transformed into the (2,5) torus knot, on the right, by changing the encircled negative crossing to a positive one.}}}\label{fig:2325}
\end{figure}

\begin{proof}[Proof of Corollary \ref{cor:2,5}] The manifold $-\Sigma(2,5,9)$ is $+1$ surgery on the $(2,5)$ torus knot of genus 2. The corresponding surgery cobordism provides the form $\langle + 1 \rangle$. As we saw at the start of this section, the canonical positive definite plumbing bounded by $-\Sigma(2,5,9)$ is isomorphic to $\Gamma_{12}$. Next, we observe from Figure \ref{fig:2325} that the $(2,5)$ torus knot is obtained from the $(2,3)$ torus knot by a positive crossing change. This induces a cobordism from +1 surgery on the latter to that of the former with intersection form $\langle + 1 \rangle$. (This is a technique used extensively in \cite{cg}.) Attaching to this cobordism the $E_8$ plumbing bounded by $-\Sigma(2,3,5)$ yields $E_8 \oplus \langle + 1 \rangle $. Finally, connect summing these three examples with copies of $\cp$ yields all lattices listed in Theorem \ref{thm:genus2}, except for $E_8$.\\

The author is aware of two constructions realizing $E_8$. The first has been communicated to the author by Motoo Tange and uses Kirby calculus. The second appears in the author's work with Golla \cite{golla-scaduto}, and uses the topology of rational cuspidal curves.
\end{proof}

The proof of Theorem \ref{thm:genus2} only uses input from instanton theory and some algebra. It is now clear that Heegaard Floer $d$-invariants were not necessary to prove Theorem \ref{thm:genus1}: starting from Theorem \ref{thm:genus2}, the extra input is our computation $f_8(\Gamma_{12})\geqslant 2$ from the previous section, and the constraint for $f_8(\mathcal{L})$ given in Corollary \ref{cor:link}.\\

On the other hand, some of the work in proving Theorem \ref{thm:genus2} may be supplemented by the $d$-invariant, as done in the previous section for Theorem \ref{thm:genus1}. Combining work of Ni and Wu \cite[Prop.1.6]{niwu} and Rasmussen \cite[Thm.2.3]{rasmussen} gives the inequalities
\begin{equation}
    0 \; \leqslant \; -d(Y)/2 \; \leqslant \; \lceil g_4(K)/2 \rceil \label{eq:dgenus}
\end{equation}
where $Y$ is $+1$ surgery on the knot $K$. If $g_4(K)=2$, then as in the case $g_4(K)=1$, the only possible non-diagonal definite lattices that can occur, up to diagonal summands, are the 14 listed in Table \ref{tab:table1}. Using Corollary \ref{cor:link} and Lemma \ref{lemma:indecomp}, of those 14 only $E_8$, $D_{12}=\Gamma_{12}$, $A_{15}$ and $O_{23}$ can occur. As already listed in the proof of Theorem \ref{thm:genus1} of the previous section, we have $f_4(A_{15}), f_4(O_{23})\geqslant 2$, both of which are special cases of the computations in the proof of Lemma \ref{lemma:f4}.

%% file: more.tex
%!TEX root = main.tex

The question of which unimodular definite lattices arise from smooth 4-manifolds with no 2-torsion in their homology bounded by a fixed homology 3-sphere $Y$ only depends on the $\Z/2$ homology cobordism class of $Y$. It is natural to wonder whether we can find linearly independent elements in the $\Z/2$ homology cobordism group $\Theta_{\Z/2}^3$ all of which bound the same set of definite unimodular lattices. If one restricts to homology cobordism classes that only bound diagonal lattices, one needs only examine the infinitely generated kernels of the invariants $d$ and $h$, for example.\\

We may then consider classes that bound the same lattices as the Poincar\'{e} sphere. For this, recall that Furuta \cite{furuta} and Fintushel and Stern \cite{fsi} used instantons to show that the family $[\Sigma(2,3,6k-1)]$ for $k\geqslant 1$ is an infinite linearly independent set in $\Theta_{\Z/2}^3$. The manifold $-\Sigma(2,3,6k-1)$ is $+1$ surgery on a genus 1 twist knot with $2k-1$ half twists. However, not all of these classes can bound the same lattices as $[\Sigma(2,3,5)]$. Indeed, the Rochlin invariant of $-\Sigma(2,3,6k-1)$ is congruent to $k$ (mod 2), so the lattice $E_8$ cannot occur when $k$ is even. In fact, here is an example where the list of lattices is the same as that of the Poincar\'{e} sphere except for $E_8$:

\vspace{.25cm}

\begin{corollary}\label{cor:2,3,11}
    If a smooth, compact, oriented and definite 4-manifold with no 2-torsion in its homology has boundary $-\Sigma(2,3,11)$, then its intersection form is equivalent to one of 
    \[
        \langle +1\rangle^{n}\;\;(n\geqslant 1),\qquad E_8\oplus \langle +1\rangle^{n}\;\;(n\geqslant 1),
    \]
    and all of these possiblities occur.
\end{corollary}

\vspace{.25cm}

There are two ways to see that $-\Sigma(2,3,11)$ bounds the lattice $E_8\oplus \langle +1 \rangle$. For one, its canonical positive definite plumbing graph is given as follows:
\begin{center}
\begin{tikzpicture}
	\draw (0,0) -- (7,0);
	\draw (2,0) -- (2,1);
	\draw[fill=black] (0,0) circle(.1);
	\draw[fill=black] (1,0) circle(.1);
	\draw[fill=black] (2,0) circle(.1);
	\draw[fill=black] (2,1) circle(.1);
	\draw[fill=black] (3,0) circle(.1);
	\draw[fill=black] (4,0) circle(.1);
	\draw[fill=black] (5,0) circle(.1);
	\draw[fill=black] (6,0) circle(.1);
	\draw[fill=black] (7,0) circle(.1);
	\node at (7,.4) {$3$};
\end{tikzpicture}
\end{center}
The unmarked nodes represent vectors of square $2$, and together form a sublattice isomorphic to $E_8$; thus the lattice must be isomorphic to $E_8\oplus \langle +1\rangle$. Alternatively, we note that the twist knot with $3$ half twists is obtained from the $(2,3)$ torus knot by a changing a positive crossing to a negative crossing, and argue as in the proof of Corollary \ref{cor:genus2}. We note that both arguments generalize to show that $-\Sigma(2,3,6k-1)$ bounds $E_8\oplus \langle +1\rangle^{k-1}$.\\

One might hope for examples of $-\Sigma(2,3,6k-1)$ bounding $E_8$ when $k$ is odd other than $k=1$. The determination of all such $k$ seems to be an open problem, but has been studied by Tange, who shows  \cite[Thm. 1.7]{tange} that this is the case for $k=3,5,\ldots,23,25$ and $k=29$.

\vspace{.25cm}

\begin{corollary}\label{cor:2,3,12n+5}
    The linearly independent elements $[\Sigma(2,3,12n+5)]\in \Theta_{\Z/2}^3$ for $ 0 \leqslant n\leqslant 12$, $n=14$ bound the same definite lattices arising from smooth 4-manifolds with no 2-torsion.
\end{corollary}

\vspace{.25cm}

Tange has informed the author that this list may be enlarged to include $n=13,15$. Yet another example that bounds the same set of lattices as the Poincar\'{e} sphere $-\Sigma(2,3,5)$ is the Brieskorn sphere $-\Sigma(3,4,7)$, whose positive definite plumbing graph has associated lattice isomorphic to $E_8$, and which is $+1$ surgery on the knot $10_{132}$ of smooth 4-ball genus 1.\\

\begin{figure}[t]
\centering
\includegraphics[scale=.8]{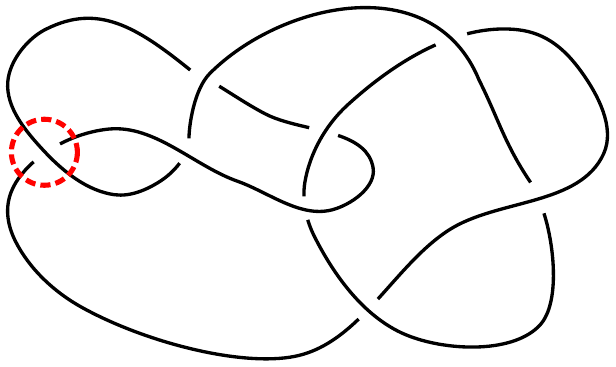}
\caption{{\small{The (3,4) torus knot, also known as $8_{19}$ in Rolfsen notation, is transformed into the (2,5) torus knot by changing the encircled positive crossing to a negative crossing.}}}\label{fig:34}
\end{figure}

In the introduction it was mentioned that $-\Sigma(3,4,11)$, obtained from $+1$ surgery on the $(3,4)$ torus knot of genus 3, is a natural candidate to consider beyond Theorem \ref{thm:genus2}. Here the Heegaard Floer $d$-invariant is $-2$, so the only possible non-diagonal reduced definite lattices that can occur are those in Table \ref{tab:table1}. We expect most of these lattices are ruled out by our obstructions. We show in \cite{golla-scaduto} that the lattices $\langle +1 \rangle$, $E_8$, $\Gamma_{12}$, $E_{7}^2$ and $A_{15}$ occur. As the proofs there use the topology of rational cuspidal curves, here we only explain how to realize $\langle +1 \rangle$, $E_8\oplus \langle + 1\rangle$, $\Gamma_{12}\oplus \langle +1 \rangle$ and $A_{15}$.\\

First, $\langle + 1\rangle$ is realized by the surgery cobordism as in all previous examples. Next, the rank 15 lattice $A_{15}$ arises as the lattice of the positive definite plumbing bounded by $-\Sigma(3,4,11)$, given by:
\begin{center}
\begin{tikzpicture}
	\draw (0,0) -- (13,0);
	\draw (3,0) -- (3,1);
	\draw[fill=black] (0,0) circle(.1);
	\draw[fill=black] (1,0) circle(.1);
	\draw[fill=black] (2,0) circle(.1);
	\draw[fill=black] (3,0) circle(.1);
	\draw[fill=black] (4,0) circle(.1);
	\draw[fill=black] (5,0) circle(.1);
	\draw[fill=black] (6,0) circle(.1);
	\draw[fill=black] (7,0) circle(.1);
	\draw[fill=black] (8,0) circle(.1);
	\draw[fill=black] (9,0) circle(.1);
	\draw[fill=black] (10,0) circle(.1);
	\draw[fill=black] (11,0) circle(.1);
	\draw[fill=black] (12,0) circle(.1);
	\draw[fill=black] (13,0) circle(.1);
	\draw[fill=black] (3,1) circle(.1);
	\node at (3.45,1.25) {$3$};
\end{tikzpicture}
\end{center}
The unmarked nodes have weight $2$. Indeed, viewing $\A_{15}$ as the subset of $\Z^{16}$ consisting of vectors whose coordinates sum to zero, the lattice $A_{15}$ may be defined as
\[
    A_{15} \; = \; \A_{15} \cup (g+\A_{15}) \cup (2g+\A_{15}) \cup (3g+\A_{15})
\]
where $g=\frac{1}{4}(-1^{12},3^{4})\in\A_{15}^\ast$ and superscripts denote repeated entries. Then the top weight 3 node in the plumbing graph represents $g$, and the other nodes are the 14 roots of $\A_{15}$ of the form $(0,\ldots,0,1,-1,0,\ldots,0)$ with the far left entry equal to zero. Finally, the lattices $E_8\oplus \langle +1 \rangle$ and $\Gamma_{12}\oplus \langle +1\rangle$ occur because there is a 2-handle cobordism from $-\Sigma(2,5,9)$ to $-\Sigma(3,4,11)$ with intersection form $\langle + 1 \rangle$. Indeed, the (3,4) torus knot is transformed into the $(2,5)$ torus knot by changing a positive crossing as in Figure \ref{fig:34}, as similarly done in the proof of Corollary \ref{cor:2,3}.\\

Finally, consider again the family $-\Sigma(2,2k-1,4k-3)$, obtained from $+1$ surgery on the family of $(2,k)$ torus knots. The initial cases $k=2$ and $k=3$ provided our main examples for Theorems \ref{thm:genus2} and \ref{thm:genus1}. The methods in this article alone seem unable to treat the general case. However, we know that the definite lattices given by
\begin{equation}
   \langle + 1 \rangle, \qquad \Gamma_{4(k-i)}\oplus \langle +1 \rangle^{i} \quad ( 0 \leqslant i \leqslant k)\label{eq:gammalist}
\end{equation}
and their sums with $\langle +1\rangle^n$ are bounded by $-\Sigma(2,2k-1,4k-3)$; the first is the surgery cobordism, and the rest follow from the fact that the $(2,k-i)$ torus knot is obtained from the $(2,k)$ torus knot by changing $i$ positive crossings. These are certainly not all the possible lattices: because $-\Sigma(2,5,9)$ bounds $\Gamma_8=E_8$, for $k\geqslant 3$ the 3-manifold $-\Sigma(2,2k-1,4k-3)$ bounds $\Gamma_8\oplus \langle +1\rangle^{k-3}$. Even if we ignore the issue of diagonal summands, the list is not complete. For example, it is shown in \cite[Prop.4.14]{golla-scaduto} that $-\Sigma(2,7,13)$ bounds the lattice $A_{15}$.

%% file: surface.tex
%!TEX root = main.tex

In this section we discuss the relations that appear in the table of Section \ref{sec:latticeterms} and in the proof of Theorem \ref{thm:char4}. In particular, we prove Proposition \ref{prop:nilppartial}, and reduce the verification of the general relations $\alpha^g \equiv 0$ (mod 2) and $\beta^{\lceil g/2 \rceil} \equiv 0 $ (mod 4) to a concrete arithmetic problem.\\

We define $\mathbb{V}_g$ to be the instanton homology, with integer coefficients, of a circle times a surface $\Sigma$ of genus $g$ with a $U(2)$-bundle that has second Stiefel-Whitney class Poincar\'{e} dual to the circle factor. More precisely, $\mathbb{V}_g$ is the $\Z/4$-graded group of Mu\~{n}oz, which is the quotient of the $\Z/8$-graded group $\mathbb{V}'_g$ by an involution $\tau$; see the discussion in \cite[\S 10]{froyshov-inequality}. Each of these is endowed with a ring structure using the maps induced by pairs of pants cobordisms times $\Sigma$. There is a map
\begin{equation}
    \Psi: \text{Sym}^\ast\left( H_0(\Sigma;\Z)\oplus H_2(\Sigma;\Z) \right) \otimes \Lambda^\ast\left( H_1(\Sigma;\Z) \right) \longrightarrow \mathbb{V}_g\label{eq:relsigma}
\end{equation}
which defines relative Donaldson invariants for the 4-manifold $\Sigma\times D^2$ with suitable bundle data. Let $x\in H_0(\Sigma;\Z)$ be the point class and $\{\gamma_i\}_{i=1}^{2g}$ be a symplectic basis of $H_1(\Sigma;\Z)$ such that $\gamma_i\cdot \gamma_{i+g}=1$. The mapping class group of $\Sigma$ acts on $\mathbb{V}_g$, and the three elements
\begin{equation}
    \alpha = 2\Psi([\Sigma]),\qquad \beta=-4\Psi(x), \qquad \gamma=-\sum_{i=1}^g \Psi(\gamma_i\gamma_{i+g})\label{eq:gens}
\end{equation}
generate the invariant part over the rationals; Mu\~{n}oz gives a presentation which is recursive in the genus \cite[\S 4]{munoz}. Our definition of $\gamma$ is one half of that from loc. cit.; see Section \ref{sec:muclasses} for this justification. The ring $\mathbb{V}'_g$ has similarly defined elements, which we denote by $\alpha',\beta',\gamma'$, of respective $\Z/8$-gradings 2, 4, 6. The involution $\tau$ acting on $\mathbb{V}_g'$ is a module homomorphism and shifts gradings by 4 (mod 8); the equivalence classes of $\alpha',\beta',\gamma'$ in $\mathbb{V}_g$ are of course $\alpha,\beta,\gamma$.

\vspace{.25cm}

\begin{lemma}\label{lemma:rels}
    Suppose a polynomial $r(\alpha,\beta,\gamma)$ is a relation in $\mathbb{V}_g$. If the corresponding polynomial $r(\alpha',\beta',\gamma')$ in $\mathbb{V}_g'$ is of homogeneous $\Z/8$-grading, then it is a relation in $\mathbb{V}_g'$.
\end{lemma}

\vspace{.15cm}

\begin{proof}
    If the quotient polynomial $r(\alpha,\beta,\gamma)$ is a relation, $r(\alpha',\beta',\gamma')=(1-\tau)\phi$ for some $\phi$ within $\mathbb{V}_g'$. Since $\tau$ is of degree 4, and $r(\alpha',\beta',\gamma')$ has homogeneous $\Z/8$-grading, $\phi=0\in \mathbb{V}_g'$.
\end{proof}

\vspace{.25cm}

When proving our inequalities, we will need to use relations in $\mathbb{V}_g'$. Lemma \ref{lemma:rels} says that so long as they are homogeneously $\Z/8$-graded in $\mathbb{V}_g'$ it suffices to show the corresponding relations in $\mathbb{V}_g$. This is the case for the relations we consider, and henceforth we restrict our attention to $\mathbb{V}_g$.\\

Let $N_g$ be the moduli space of projectively flat connections on a $U(2)$-bundle with fixed odd determinant over a surface of genus $g$. Mu\~{n}oz's work shows that $\mathbb{V}_g\otimes\C$ is isomorphic to $H^\ast(N_g;\C)$, and in fact the ring structure of the former is a deformation of the latter. More precisely, the product in $\mathbb{V}_g$ is equal to the cup product in $H^\ast(N_g;\C)$ up to lower order terms of equal mod 4 gradings. Furthermore, the isomorphism is well-defined over the rationals, so we may replace $\C$ by $\Q$. There is also a Morse-Bott spectral sequence, due to Fukaya \cite{fukaya}, starting at $H^\ast(N_g;\Z)$ and converging to $\mathbb{V}_g$. Since $H^\ast(N_g;\Z)$ is torsion-free, as proven by Atiyah and Bott \cite[Thm. 9.10]{ab}, and the spectral sequence collapses over $\Q$, it must collapse for all coefficient fields. Thus we obtain

\vspace{.25cm}

\begin{prop}
    $\mathbb{V}_g$ is torsion-free.\label{prop:torsion}
\end{prop}

\vspace{.25cm}

However, the ring structure of $\mathbb{V}_g$ is substantially more complicated than that of $\mathbb{V}_g\otimes\Q$, since $\alpha,\beta,\gamma$ do no generate the invariant part of $\mathbb{V}_g$. This is already true for $H^\ast(N_g;\Z)$, which requires more generators than does $H^\ast(N_g;\Q)$, see \cite[\S 9]{ab}. Nonetheless, the relations of interest can be extracted from Mu\~{n}oz's presentation, which we now recall: set $\zeta_0=1$, and recursively define
\begin{equation*}
    \zeta_{r+1}=\alpha\zeta_r+r^2(\beta+(-1)^r8)\zeta_{r-1}+4r(r-1)\gamma\zeta_{r-2}.
\end{equation*}
Each $\zeta_r=\zeta_r(\alpha,\beta,\gamma)$ is a polynomial in the variables $\alpha,\beta,\gamma$ with integer coefficients. Then the ideal $(\zeta_g,\zeta_{g+1},\zeta_{g+2})$ is a complete set of relations for the part of $\mathbb{V}_g\otimes\Q$ invariant under the mapping class group action, see \cite[Thm.16, Prop.20]{munoz}.

\vspace{.25cm}

\begin{lemma}\label{lemma:mod8}
    $\beta\equiv \alpha^2$ {\emph{(mod 8)}}.
\end{lemma}

\vspace{.15cm}

\begin{proof}
   The corresponding relation holds in $H^4(N_g;\Z)$. Indeed, the degree 4 element
   \[
         {g-1 \choose 2}\alpha^2 + (2g-1)\frac{\alpha^2-\beta}{8}
   \]
   is integral, see \cite[Eq.(7), Prop.2.4]{ss}; it is the second Chern class of the push-forward of a universal bundle. Multiplying both sides by $8$, and a mod 8 inverse for $(2g-1)$, yields $\beta \equiv \alpha^2$ (mod 8) in the ring $H^4(N_g;\Z)$. Since the product in $\mathbb{V}_g$ is a deformation of the product in $H^\ast(N_g;\Z)$ respecting mod 4 gradings, within $\mathbb{V}_g$ we have $\alpha^2-\beta+c\equiv 0$ (mod 8), where $c$ is some constant. There is a map $r_k:\mathbb{V}_g\to\mathbb{V}_{g-k}$ induced by a cobordism which contracts $k$ handles, cf. \cite[Lemma 9]{munoz}. For $g\geqslant 1$, we have $0\equiv r_{g-1}(\alpha^2-\beta+c)\equiv c$ (mod 8) since in $\mathbb{V}_1$ the relations $\alpha=0$ and $\beta=8$ follow from $\zeta_1$ and $\zeta_2$. Thus $c\equiv 0$ (mod 8) and the relation follows.
\end{proof}

\vspace{.25cm}

This lemma allows us to write $\beta=\alpha^2+8\varepsilon$ for some element $\varepsilon\in\mathbb{V}_g$. Define the double factorial $n!!=n(n-2)(n-4)\cdots 1$ for $n>0$ odd. We propose the following.

\vspace{.25cm}

\begin{conjecture}\label{conj:alpha}
    $(2g-3)!!\zeta_g(\alpha,\alpha^2+8\varepsilon,\gamma)/g!$ is a polynomial in $\alpha,\varepsilon, \gamma$ with integer coefficients. Furthermore, the reduction of this polynomial mod 4 is congruent to $\pm\alpha^g$.
\end{conjecture}

\vspace{.25cm}

The verification of this conjecture implies the relations $\alpha^g\equiv 0$ (mod 2) and $\beta^{\lceil g/2 \rceil} \equiv 0$ (mod 4) within $\mathbb{V}_g$. Indeed, the polynomial in the conjecture is a relation in $\mathbb{V}_g$, since according to Mu\~{n}oz it is a relation in $\mathbb{V}_g\otimes \Q$, and $\mathbb{V}_g$ is torsion-free. Its reduction modulo 4 implies the relation $\alpha^g \equiv 0$ (mod 4), which by Lemma \ref{lemma:mod8} implies the two desired relations.\\

In fact, for some of our applications in mind, less is required. For example, we consider what is required to prove the conjectural inequality obtained from Theorem \ref{thm:char4} by replacing the left hand side with $ \lceil g/2 \rceil$. Define $N_\beta^4(g)'$ to be the nilpotency degree of $\beta$ in the ring $(\overline{\mathbb{V}}_g/\text{Tor})\otimes \Z/4$. The argument given for Theorem \ref{thm:char4} in the next section is easily seen to work for $N^4_\beta(g)'$ in place of $N_\beta^4(g)$. Of course $N^4_\beta(g)'\leqslant N_\beta^4(g)$, although they are likely equal. The point of using this alternative definition is as follows. We may set $\gamma=0$ in the recursive equation to define $\zeta'_{r+1}=\zeta'_r + r^2(\beta+(-1)^r8)\zeta'_{r-1}$ with $\zeta'_0=1$. Then $\xi_r=\zeta'_r(\alpha,\alpha^2+8\varepsilon)$ is a polynomial only in $\alpha$ and $\varepsilon$. As we only are concerned with relations mod 4, to prove that $N_\beta^4(g)'=\lceil g/2 \rceil$ it suffices to show that the rational coefficients of $\xi_r/r!$ all have reduced fraction forms with odd denominators, and have numerators divisible by 4, except for the coefficient in front of $\alpha^g$, which should be odd.

\vspace{.25cm}

\begin{proof}[Proof of Proposition \ref{prop:nilppartial}] The first few instances of Conjecture \ref{conj:alpha} are verified by hand, and we verify the rest of the cases $g\leqslant 128$ by computer. The first 8 polynomials defined in Conjecture \ref{conj:alpha} are given in Table 3 for illustration.
\end{proof}

\vspace{.25cm}

Finally, we remark that $\alpha^g \equiv 0$ (mod 2) is a relation in the ring $H^\ast(N_g;\Z/2)$ by work of the author with M. Stoffregen \cite{ss}. The above scheme suggests an alternative route to proving this relation. Indeed, the ring $H^\ast(N_g;\Q)$ has its own recursive presentation, which inspired the work of Mu\~{n}oz; in the recursive definition of $\zeta_r$ above, simply remove the term $(-1)^r8$. Then Conjecture \ref{conj:alpha} may be formulated with these modified polynomials. In particular, we suspect that the relation $\alpha^g\equiv 0$ (mod 4) also holds in $H^\ast(N_g;\Z/4)$.\\

In \cite{ss} it is also proven that $\alpha^{g-1}\not\equiv 0$ (mod 2) within $H^\ast(N_g;\Z/2)$. This implies $N_\alpha^2(g)\geqslant g$, which aligns with the first part of Lemma \ref{prop:gamma} and inequality \eqref{eq:mod2speculation}. Indeed, since $\mathbb{V}_g\otimes \Z/2$ is a deformation of the ring $H^\ast(N_g;\Z/2)$, the deformations being of lower degree but homogeneous mod 4, then because $\alpha^{g-1}$ is nonzero in the latter, it must also be so in the former.

%% file: zetatable.tex
\begin{figure}[t]
\begin{center}
Table 3\\
{\small{
{\renewcommand{\arraystretch}{1.5}
\begin{tabular}{p{.3cm}|p{12cm}}
$g$ & $(2g-3)!!\zeta_g(\alpha,\alpha^2+8\varepsilon,\gamma)/g!$\\
\hline
$ 1 $ & $ \alpha $\\
$ 2 $ & $ \alpha^{2} + 4 \varepsilon - 4 $\\
$ 3 $ & $ 3 \alpha^{3} + 20 \alpha \varepsilon + 12 \alpha + 4 \gamma $\\
$ 4 $ & $ 15 \alpha^{4} + 160 \alpha^{2} \varepsilon - 120 \alpha^{2} + 20 \alpha \gamma + 360 \varepsilon^{2} - 720 \varepsilon + 360 $\\
$ 5 $ & $ 105 \alpha^{5} + 1456 \alpha^{3} \varepsilon + 840 \alpha^{3} + 224 \alpha^{2} \gamma + 4984 \alpha \varepsilon^{2} + 6160 \alpha \varepsilon + 1232 \gamma \varepsilon + 3192 \alpha + 560 \gamma $\\
$ 6 $ & $ 945 \alpha^{6} + 16884 \alpha^{4} \varepsilon - 11340 \alpha^{4} + 2016 \alpha^{3} \gamma + 93576 \alpha^{2} \varepsilon^{2} - 146160 \alpha^{2} \varepsilon + 14448 \alpha \gamma \varepsilon + 151200 \varepsilon^{3} + 74088 \alpha^{2} - 5040 \alpha \gamma + 840 \gamma^{2} - 453600 \varepsilon^{2} + 453600 \varepsilon - 151200 $\\
$ 7 $ & $ 10395 \alpha^{7} + 221364 \alpha^{5} \varepsilon + 124740 \alpha^{5} + 28116 \alpha^{4} \gamma + 1558392 \alpha^{3} \varepsilon^{2} + 1851696 \alpha^{3} \varepsilon + 342672 \alpha^{2} \gamma \varepsilon + 3621024 \alpha \varepsilon^{3} + 957528 \alpha^{3} + 144144 \alpha^{2} \gamma + 9240 \alpha \gamma^{2} + 6852384 \alpha \varepsilon^{2} + 978912 \gamma \varepsilon^{2} + 7061472 \alpha \varepsilon + 931392 \gamma \varepsilon + 1929312 \alpha + 522720 \gamma $\\
$ 8 $ & $ 135135 \alpha^{8} + 3418272 \alpha^{6} \varepsilon - 2162160 \alpha^{6} + 365508 \alpha^{5} \gamma + 31141968 \alpha^{4} \varepsilon^{2} - 43531488 \alpha^{4} \varepsilon + 5319600 \alpha^{3} \gamma \varepsilon + 118472640 \alpha^{2} \varepsilon^{3} + 22177584 \alpha^{4} - 1873872 \alpha^{3} \gamma + 264264 \alpha^{2} \gamma^{2} - 285597312 \alpha^{2} \varepsilon^{2} + 19260384 \alpha \gamma \varepsilon^{2} + 151351200 \varepsilon^{4} + 288699840 \alpha^{2} \varepsilon - 14030016 \alpha \gamma \varepsilon + 1633632 \gamma^{2} \varepsilon - 605404800 \varepsilon^{3} - 89945856 \alpha^{2} + 7948512 \alpha \gamma - 480480 \gamma^{2} + 908107200 \varepsilon^{2} - 605404800 \varepsilon + 151351200 $
\end{tabular}
}
}}
\end{center}
\end{figure}

%% file: proofsineq.tex
%!TEX root = main.tex

We now proceed to the proofs of Theorem \ref{thm:char4} and Proposition \ref{prop:mod8}. These are adaptations of Fr\o yshov's argument as given in \cite{froyshov-inequality}, which we closely follow and modify accordingly to our choices of coefficient rings. For most of the technical details we refer to loc. cit. In the final subsection we discuss some other adaptations.

\subsection{Proofs of Theorem \ref{thm:char4} and Proposition \ref{prop:mod8}}\label{subsec:proofs}

Let $X$ be a smooth, closed, oriented 4-manifold. For now we also assume $b_1(X)=0$. Suppose $b_2^+(X)=n\geqslant 1$ and let $\Sigma_1,\ldots,\Sigma_n$ be pairwise disjoint connected and oriented embedded surfaces with $\Sigma_i$ of genus $g_i$ such that $\Sigma_i\cdot \Sigma_i = 1$. We will eventually specialize to the case $n=1$. Let $W$ be the result of replacing a tubular neighborhood $U_i$ of $\Sigma_i\subset X$ with $U_i\#\cpbar$ for each $i$. Upon orienting the exceptional sphere $S_i$ in the corresponding copy of $\overline{\cp}$, we form two internal connected sums $\Sigma_i^\pm$ between $\Sigma_i$ and $S_i$, one preserving the orientation of $S_i$, the other reversing. Now define a smooth $n$-dimensional family of metrics $g(t)$ where $t=(t_1,\ldots,t_n)\in \R^n$ on the closed 4-manifold $W$, which as $t_i\to \pm \infty$ stretches along a link of $\Sigma_i^\pm$. Since $\Sigma_i^\pm\cdot \Sigma_i^\pm=0$, each such link may be identified with $S^1\times\Sigma_i^\pm$.\\

Let $s_i=\text{PD}[S_i]$, and let $E_k\to W$ be the $U(2)$-bundle with $c_1(E_k)=w+\sum s_i$ and $c_2(E_k)=k$. Write $\mathcal{L}\subset H^2(X;\Z)/\text{Tor}$ for the lattice of vectors vanishing on the $[\Sigma_i]$. We choose $w$ so that modulo torsion it is an element in $\mathcal{L}$ which is extremal. Denote by $M_{k,t}$ the moduli space of projectively $g(t)$-anti-self-dual connections on $E_k$, and let $\mathscr{M}_k$ denote the disjoint union of $M_{k,t}$ over $t\in \R^n$. After perturbing, the irreducible stratum $\mathscr{M}_k^\ast\subset \mathscr{M}_k$ is a smooth and possibly non-compact manifold of dimension $8c_2-2c_1^2-3(1-b_1+b_2^+)+n$. Thus
\begin{equation}
    \dim \mathscr{M}_k^\ast \; = \; 8k + 2|w^2| - 3.\label{eq:dim}
\end{equation}
If $k<0$, then $\mathscr{M}_k$ has no reducibles, while $\mathscr{M}_0$ contains a finite number. Denote by $\mathscr{M}'_k$ the result of removing small neighborhoods of each reducible; in particular, $\mathscr{M}'_k=\mathscr{M}_k$ if $k<0$. The assumption that $w$ is extremal rules out bubbling off of reducible solutions in these moduli spaces.\\

Recall from \cite[\S 5.1.2]{dk} that the $\mu$-map is given by
\begin{equation}
    \mu:H_i(W;\Q) \longrightarrow H^{4-i}(\mathscr{B}^\ast_E;\Q), \qquad \mu(a)= -\frac{1}{4}p_1(\mathbb{E})/a.
\end{equation}
Here $E$ is a $U(2)$-bundle over a 4-manifold $W$, and $\mathbb{E}$ is the universal adjoint $SO(3)$ bundle over $\mathscr{B}^\ast_E\times W$, where $\mathscr{B}^\ast_E$ is the configuration space of connections on $E$. The basepoint fibration associated to $x\in W$ is the restriction of $\mathbb{E}$ to a slice $\mathscr{B}_E^\ast\times\{x\}$. For later use, we also introduce notation for the second Stiefel--Whitney class: 
\begin{equation}
    \nu(x)\; = \; w_2(\mathbb{E})/1  \in H^{2}(\mathscr{B}^\ast_E;\Z/2).
\end{equation}
When defining (relative) Donaldson invariants on 4-manifolds, one cuts down moduli spaces inside $\mathscr{B}^\ast_E$ using geometrically constructed divisors representing $\mu$-classes. Henceforth we write $x\in H_0(W;\Z)$ for the point class.\\

Returning to our above setup, to any $a_1,\ldots,a_k\in \{x\}\cup H_2(W;\Z)$ which descend to $\mathcal{L}^\ast$, subset $S\subset \mathscr{M}_k'$, and nonnegative integers $j_i \geqslant 0$ for $i=1,\ldots,k$, we use the shorthand $\mu(a_1)^{j_1}\cdots\mu(a_k)^{j_k} S$ for the intersection of $S$ with $j_i$ generic geometric representatives for $\mu(a_i)=-p_1(\mathbb{E})/4a_i$ supported away from where $g(t)$ varies, as $i$ runs over $1,\ldots,k$. Also, let $\mu(x)_{i}^jS$ denote the intersection of $S$ with a geometric representative depending on $t$ for $-p_1(\mathbb{E})/4\text{pt}$, where the basepoint is in the location of the stretched link of $\Sigma_i^\pm$ as $t_i\to\pm \infty$. For the constructions see \cite[\S 7]{froyshov-inequality}, where $\mu(x)_i$ is called $x_i$. 
The following lemma is for a general 4-manifold $W$ with $U(2)$ bundle $E$.

\vspace{0.25cm}

\begin{lemma}\label{lemma:2div}
    For $a\in H_2(W;\Z)$,  $2\mu(a)$ defines a class in $H^2(\mathscr{B}^\ast_E;\Z)$. If further $\langle w_2(E), a\rangle \equiv 0$ \text{\emph{(mod 2)}}, then $\mu(a)$ defines a class in $H^2(\mathscr{B}^\ast_E;\Z)$.
\end{lemma}

\vspace{0.15cm}

\begin{proof}
This follows from \cite[Lemma 3]{amr}. The proof is short so we include it. If $\langle w_2(E) ,a\rangle\equiv 0$, we can lift $\mathbb{E}\to \mathscr{B}^\ast_E\times W$ to a $U(2)$ bundle $\mathbb{F}$ such that $\langle c_1(\mathbb{F}),a\rangle = 0$. Now use $-p_1(\mathbb{E})=4c_2(\mathbb{F})-c_1^2(\mathbb{F})$ and $c_1(\mathbb{F})=c_1(\mathbb{F}|_{\mathscr{B}^\ast})\times 1 + 1\times c_1(E)$ to compute $-p_1(\mathbb{E})/a=4c_2(\mathbb{F})/a$. In general, $2\mu(a)$ is integral, as $-p_1(\mathbb{E})/a=4c_2(\mathbb{F})/a-2c_1(\mathbb{F}|_{\mathscr{B}^\ast})\times c_1(E)/a$ is even.
\end{proof}

\vspace{0.25cm}

This lemma reduces to Corollary 5.2.7 of \cite{dk} when $w_2(E)\equiv 0$ (mod 2). When cutting down moduli spaces by $\mu(a)$, for $a$ as in the lemma, the geometric representatives we have in mind are those constructed as in loc. cit. using line bundles of coupled Dirac operators over the surface.

\vspace{0.25cm}

\begin{proof}[Proof of Theorem \ref{thm:char4}]
Assume the setup above. Let $a_1,\ldots,a_{m_0}, a'_1,\ldots,a'_{m_1}\in  H_2(W;\Z)$ be such that each class descends to $\mathcal{L}^\ast$, and $\langle w,a_i \rangle \equiv 0$ (mod 2) for each $a_i$. Further, setting $m=m_0+m_1$, assume $w^2\equiv m$ (mod 2). Lemma \ref{lemma:2div} says each $\mu(a_i)$ and $2\mu(a_i')$ are integral. Suppose as in the definition of $f_4(\mathcal{L})$ that $2^{-m_0}\eta(\mathcal{L},w,a,m)\not\equiv 0$ (mod 4), where $a=a_1\cdots a_{m_0}a'_1\cdots a_{m_1}'$. Suppose for contradiction that the following inequality holds:
\begin{equation}
   \sum_{i=1}^n N_\beta^4(g_i) \; < \; (|w^2|-m)/2.\label{eq:mainmech4}
\end{equation}
Set $n_i= N_\beta^4(g_i)$.  We define the following smooth, orientable 1-manifold with boundary and a finite number of non-compact ends:
\begin{equation}
    \hat{\mathscr{M}} \; := \; \left(4\mu(x)\right)^{(|w^2|-m)/2-1-\sum n_i}\prod_{k=1}^{m_0}\mu(a_k)\prod_{j=1}^{m_1}2\mu(a_j')\prod_{i=1}^n \left(4\mu(x)_i\right)^{n_i} \mathscr{M}'_0\label{eq:thespace4}
\end{equation}
Here it is important that we cut down by divisors associated to integral cohomology classes. The boundary points of $\hat{\mathscr{M}}$ arise from the deleted neighborhoods of reducibles in $\mathscr{M}_0$. Denote by $\mathcal{T}$ the torsion subgroup of $H^2(X;\Z)$. Then each pair $\{z,-z\}\subset \text{Min}(w+2\mathcal{L})$ corresponds to $2^{n}\cdot\#\mathcal{T}$ many reducibles. Indeed, as shown in \cite[Lemma 1]{froyshov-inequality}, the family of metrics may be chosen such that the reducibles in $\mathscr{M}_0$ correspond to splittings $E_0=L_1\oplus L_2$ into line bundles such that
\begin{equation}\label{eq:truereducibles}
	c_1(L_1\otimes L_2^{-1}) = \overline{z} - \sum_{i=1}^{n} \epsilon_i s_i
\end{equation}
where each $\epsilon_i=\pm 1$ and the projection of $\overline{z}$ to $H^2(W;\Z)/\mathcal{T}$ is some $z\in \text{Min}(w+2\mathcal{L})$. The symmetry of swapping $L_1$, $L_2$ in \eqref{eq:truereducibles} leads to the consideration of the pair $\{z,-z\}$, and for each such pair the freedom of torsion multiplies the number of reducibles by $\#\mathcal{T}$, while the possibilities for the signs $\epsilon_i$ multiply the number by $2^{n}$.\\

The neighborhood of each reducible in $\mathscr{M}_0$ is a cone on a complex projective space of dimension $d=2|w^2|-4$. To a reducible associated to a class $r$ as in \eqref{eq:truereducibles}, let $h$ be the degree two generator of the projective space $\mathbb{C}\mathbb{P}^d$ which is the link of this cone. Then by \cite[Prop. 5.1.21]{dk} we have
\begin{equation}
   \mu(x)|_{\mathbb{C}\mathbb{P}^d} \; = \; \pm \frac{1}{4}h^2, \qquad \mu(a_i)|_{\mathbb{C}\mathbb{P}^d} \; =  \pm\frac{1}{2}\langle r, a_i\rangle h \label{eq:links}
\end{equation}
where the second relation of course also holds for $a_i'$. Fr\o yshov shows that the moduli space can be oriented such that each of the $2^{n}\cdot \#\mathcal{T}$ reducibles in \eqref{eq:truereducibles} associated to $\{z,-z\}\subset \text{Min}(w+2\mathcal{L})$ have the same orientations for their cones. We compute
\begin{equation}
   \#\partial \hat{\mathscr{M}} = 2^{n-{m_0}}\cdot \#\mathcal{T}\cdot\, \eta(\mathcal{L},w,a,m), \label{eq:boundary}
\end{equation}
which is a rescaling of \cite[Proposition 5]{froyshov-inequality}.\\

Now we discuss the ends of the moduli space (\ref{eq:thespace4}). These arise as the metric family parameters $t_i$ go off to $\pm\infty$. Transversality ensures that at most one such parameter can stay unbounded for a given sequence of instantons in $\hat{\mathscr{M}}$. The part of $\hat{\mathscr{M}}$ with fixed $\pm t_i = \tau \gg 0$ is a finite number of points, which by gluing theory is a pair of instantons over $\R^2\times \Sigma_i^\pm$ and over $W\setminus \Sigma_i^\pm$. We may write 
\begin{equation}
   \# \hat{\mathscr{M}}_{\pm t_i = \tau} \; = \; \phi_i^\pm \cdot \psi_i^\pm\label{eq:ends}
\end{equation}
where $\phi_i^\pm\in \mathbb{V}'_{g_i}$ counts instantons over $\R^2\times\Sigma_i^\pm$, and $\psi_i^\pm\in (\mathbb{V}_{g_i}')^\ast$ over $W\setminus \Sigma_i^\pm$ in the family of metrics with $\pm t_i = \tau$ fixed. Here $\mathbb{V}_{g}'$ is the $\Z/8$-graded instanton cohomology of a circle times a surface as discussed in Section \ref{sec:surface}.\\

A priori, the elements $\phi_i^\pm$ and $\psi_i^\pm$ only define cochains in their Floer cochain complexes. The unperturbed Chern-Simons functional for the restricted bundle over $S^1\times\Sigma_i^\pm$ is Morse-Bott along its critical set, which is two copies of $N_{g}$ where $g=g_i$. According to Thaddeus \cite{thaddeus}, the manifold $N_g$ has a perfect Morse function. We perturb the Chern-Simons functional so that its critical set consists of two copies of the critical points of such a function. (See for example \cite[Proposition 6]{bd} and the surrounding discussion for this sort of perturbation.) The rank of the instanton Floer cochain complex coincides with that of $\mathbb{V}_g$, and so has zero differential. Thus $\phi_i^\pm$ and $\psi_i^\pm$ may also be viewed as Floer cohomology classes, as claimed in the previous paragraph. In this way we remove the restriction in \cite{froyshov-inequality} that all but one of the surfaces have genus 1.\\

Now we make a futher simplification which effectively removes a factor of $2$ appearing in \eqref{eq:boundary}. Fix $1\leqslant i\leqslant n$. Let $\rho_i$ be a diffeomorphism of $W$ which is reflection in the exceptional class in $\cpbar\setminus B^4\subset U_i\#\cpbar$ and is the identity elsewhere. We arrange that $\rho_i$ interchanges $\Sigma_i^+$ and $\Sigma_i^-$, and fixes all other $\Sigma_j^\pm$. We may choose the family of metrics such that for $|t_i|\geqslant \tau$, the map $\rho_i$ interchanges the metrics $g(t_1,\ldots,t_i,\ldots,t_n)$ and $g(t_1,\ldots,-t_i,\ldots,t_n)$. Consequently, $\rho_i$ interchanges $\mathscr{M}_{ +t_i \geqslant  \tau}$ and $\mathscr{M}_{ - t_i \geqslant  \tau}$ in an orientation-preserving fashion. To see this last point, we note that $\rho_i^\ast$ preserves the orientation of $H^+(W;\R)$, reverses the orientation of the metric family, and reverses the orientation rule for the moduli space with a fixed metric constructed in \cite{donaldson-orientation}:
\[
	(-1)^{\left(\frac{\rho_i^\ast(c_1(E_0))-c_1(E_0)}{2}\right)^2}= (-1)^{s_i^2}=-1.
\]
We may arrange that cutting down the moduli space by the divisors is also compatible with $\rho_i$, so that $\rho_i$ also interchanges $\phi_i^+\cdot \psi_i^+$ and  $\phi_i^-\cdot \psi_i^-$ in a sign-preserving way. The number of ends of $\hat{\mathscr{M}}$ is then counted to be the following:
\begin{equation}\label{eq:phipsiend}
	\sum_{i=1}^n \phi_i^-\cdot \psi^-_i + \phi_i^+\cdot \psi_i^+ = 2\sum_{i=1}^n  \phi_i^+\cdot \psi_i^+.
\end{equation}
The number of boundary points and the number of ends of the 1-manifold $\hat{\mathscr{M}}$ counted with signs must be zero, and so from \eqref{eq:boundary} and \eqref{eq:phipsiend} we obtain the relation
\[
	2^{n-{m_0}}\cdot\#\mathcal{T}\cdot\, \eta(\mathcal{L},w,a,m) +  2\sum_{i=1}^n  \phi_i^+\cdot \psi_i^+ = 0
\]
Now take $n=1$ and divide this relation by $2$. Then we have
\begin{equation}\label{eq:finalcounts}
	2^{-{m_0}}\cdot\#\mathcal{T}\cdot\, \eta(\mathcal{L},w,a,m) +  \phi_1^+\cdot \psi_1^+ = 0
\end{equation}
The class $\phi^+_1$ comes from $4\mu(x)_1^{n_1}$ in the expression (\ref{eq:thespace4}), and so $\phi_1^+ = (\beta')^{n_1}$, in the notation of Section \ref{sec:surface}. By the definition of $n_1=N_\beta^4(g_1)$ and Lemma \ref{lemma:rels}, the element $\phi_1^+$ is in the ideal of $\mathbb{V}_g'\otimes \Z/4$ generated by $\mu$-classes of loops. Here $g=g_1$. Similar to the argument of \cite[\S 10]{froyshov-inequality}, we conclude $ \phi_1^+\cdot \psi_1^+$ vanishes mod 4, essentially because (relative) Donaldson invariants involving $\mu$-classes of loops vanish for 4-manifolds with $b_1=0$; see Section \ref{sec:muclasses} for this justification. But the left term in \eqref{eq:finalcounts} is by assumption non-zero mod $4$, a contradiction.\\

We make two final remarks. First, although we worked with a homogeneous element $a=a_1\cdots a_{m_0}a'_1\cdots a_{m_1}'$ such that $\langle a_i,w\rangle \equiv 0 \pmod 2$, the argument easily extends to any linear combination of such elements. This allows the argument to go through for all the data included in the definition of $f_4(\mathcal{L})$. Second, the general case reduces to that of $b_1(X)=0$ by surgering loops as in \cite[Prop.2]{froyshov-inequality}.
\end{proof}

\vspace{0.25cm}

\begin{proof}[Proof of Proposition \ref{prop:mod8}]
The proof is almost the same, except that every instance of mod $4$ coefficients is replaced with mod $8$ coefficients. We are led to the inequality $N^8_\beta(g)\geqslant f_8(\mathcal{L})$, and $N^8_\beta(1)=1$ because $\beta$ is multiplication by $8$ on $\mathbb{V}_1$.
\end{proof}

\vspace{0.25cm}

As the proof of Theorem \ref{thm:char4} is not particular to $\mathbb{Z}/4$, we may also apply it to the case of $\mathbb{Z}/2$ coefficients. However, our computations suggest that $N^4_\beta(g)=N^2_\beta(g)$, in which case the resulting inequality is implied by Theorem \ref{thm:char4}.\\

On the other hand, in the setting of mod $2$ coefficients, we may replace the role of $\beta$ with $\alpha$, which is implemented by replacing every instance of the class $4\mu(x)$ with $\nu(x)$, including the metric-dependent divisors. However, in this case, the cut down moduli space is not a priori naturally oriented, and the division by $2$ in obtaining \eqref{eq:finalcounts} is problematic. Perhaps further insights or other methods can overcome this obstacle. Nonetheless, in Section \ref{sec:alt} we will exhibit instances where cutting down by the second Stiefel--Whitney class gives rise to useful inequalities in the setting of mod $2$ instanton homology for homology $3$-spheres.

\subsection{$\mu$-classes of loops}\label{sec:muclasses}
We now take a moment to make more precise which geometric representatives for $\mu$-classes of loops are to be used in the above constructions. We refer to the simplified situation described in \cite[Sec.11]{froyshov-inequality}. There, a Riemannian 4-manifold $X$ with tublular end $[0,\infty)\times Y$ is considered, equipped with a $U(2)$-bundle that restricts to some oriented surface non-trivially within the tubular end. Fix a loop $\lambda:S^1\to X$. Following constructions from \cite{km-embedded}, Fr\o yshov then associates to $\lambda$ three classes $\Phi,\Psi_+,\Psi_-\in I^\ast(P;\Z)$ in the instanton Floer cohomology of $P\to Y$, the restriction of the bundle over $X$ to $Y$. Roughly, $\Phi$ cuts down moduli by the locus of connection classes with holonomy $1\in SO(3)$, and $\Psi_\pm$ cuts down by holonomy $\pm 1 \in SU(2)$. These classes satisfy the relation $\Phi = \Psi_+ + \Psi_-$. It is observed in \cite[Sec.11]{froyshov-inequality} that $\Psi_+ = \Psi_- $ and $\Phi = 2\Psi_\pm$ modulo 2-torsion. However, in our constructions above, $I^\ast(P;\Z)$ arises as $\mathbb{V}_g'$, which is torsion-free. Thus $\Phi/2=\Psi_\pm$ is an unambiguously defined class over the integers, and is the one which we have in mind when cutting down by $\mu$-classes of loops over arbitrary coefficient rings.\\

According to \cite[\S 2(ii)]{km-embedded}, with rational coefficients $\Phi$ is equal to what is usually denoted $2\mu(\lambda)$. Thus $\Psi_\pm$ is an integral class that agrees with $\mu(\lambda)$, the latter, in general, a priori only defined over the rationals. The map $\Psi$ of (\ref{eq:relsigma}) on a 1-dimensional homology class $[\lambda]$ is now more precisely defined using $\Psi_\pm=\Psi_\pm(\lambda)$, from the 4-manifold $D^2\times \Sigma$ with appropriate bundle. The independence of the chosen representative $\lambda$ for the class follows from \cite[Prop.7]{froyshov-inequality}. We have now justified our claim, in Section \ref{sec:surface}, that the class $\gamma$, as we have normalized it, is integral.\\

We can now also be more precise about the definition of the ring $\overline{\mathbb{V}}_g$ from Sections \ref{sec:latticeterms} and \ref{sec:surface}: it is the quotient of $\mathbb{V}_g$ by the ideal generated by elements $\Psi_\pm(\lambda)=\Psi(\lambda)$, defined using the 4-manifold $D^2\times \Sigma$ with appropriate bundle, and allowing $\lambda$ to range over a symplectic basis of loops $\{\gamma_i\}$ for the surface $\Sigma$. In particular, this ideal contains $\gamma$.\\

Finally, we note that with these conventions the proof that $\phi_i^\pm$ vanishes mod $4$ in the proof of Theorem \ref{thm:char4} now adapts from the argument in \cite{froyshov-inequality}: by definition of $N^4_\alpha(g)$, we have $\phi_i^+ \equiv \sum \Psi_\pm(\lambda_i)\chi_i$ (mod 4) for some loops $\lambda_i$ in the 4-manifold at hand, and from \cite[Prop.7]{froyshov-inequality} the latter quantity vanishes. Indeed, in our proof it is assumed that $b_1(X)=0$ and thus $\lambda_i$ torsion; then $\Psi_\pm(\lambda_i)$ is torsion in $\mathbb{V}_g$, so must be zero.

\subsection{Other adaptations}\label{sec:oddchar}

Let us first compare the above arguments to that of Theorem \ref{thm:char0}. We return to the setup in the proof of Theorem \ref{thm:char4}, before specializing to $n=1$. Set $n_i=N_\beta^0(g_i)$. We then consider the $1$-dimensional part of the $\Q$ linear combination of oriented manifolds
\begin{equation}
    \mu(x)^{(|w^2|-m)/2-1-2\sum n_i} \prod_{j=1}^{m}\mu(a_j) \prod_{i=1}^n \left(\mu(x)_i^2-64\right)^{n_i} \sum_{k\leqslant 0}\mathscr{M}'_{k}\label{eq:thespace0}
\end{equation}
The number of boundary points, which only appear within $\mathscr{M}'_0$, is equal to a power of two times $\#\mathcal{T}\cdot\eta(\mathcal{L},w,a,m)$, while the number of ends is zero. Here $a=a_1\cdots a_m$ where each $a_j$ mod torsion is in $\mathcal{L}^\ast$. With these modifications, the argument is much the same as before. This handles the case of Theorem \ref{thm:char0} for a closed 4-manifold. The more general case follows from this with minor modifications as in \cite{froyshov-inequality}. Note that we have slightly improved Fr\o yshov's Theorem 2 from \cite{froyshov-inequality} by removing the restriction that all but one of the surfaces has genus $1$.\\

The above argument is also easily adapted to the case in which $\Q$ is replaced by $\Z/p$ for $p$ an odd integer. To begin, we define
\begin{equation*}
  N^p_\beta(g) \;:=\; \min\left\{ n\geqslant 1:\; (\beta^2-64)^n \equiv 0 \in \overline{\mathbb{V}}_g\otimes \Z/p \right\}
\end{equation*}
for $g\geqslant 1$ and $N_\beta^p(0)=0$. Upon setting $n_i=N_\beta^p(g_i)$, we may consider the 1-dimensional part of (\ref{eq:thespace0}) a formal $\Z/p$ linear combination of 1-manifolds; the powers of two in the definitions of the $\mu$-classes are invertible modulo $p$. The number of boundary points is again $\#\mathcal{T}\cdot\eta(\mathcal{L},w,a,m)$ up to a power of two, and the number of ends is zero mod $p$. Define $e_p(\mathcal{L})$ by modifying the condition in the definition of $e_0(\mathcal{L})$ that $\eta\neq 0$ to $\eta\not\equiv 0$ (mod p). Then under the hypotheses of Theorem \ref{thm:char0}, if $\#\mathcal{T}$ is relatively prime to the odd integer $p$, we obtain
    \begin{equation*}
        \sum_{i=1}^n N_\beta^p(g_i) \; \geqslant \; e_p(\mathcal{L}).
    \end{equation*}
Furthermore, if $p$ is prime, and the 4-manifold has instead a homology 3-sphere boundary $Y$, then the same inequality holds upon adding to the left side $h_p(Y)$, Fr\o yshov's instanton invariant defined over $\Z/p$. The modifications needed to deduce the case with a homology 3-sphere boundary from the closed 4-manifold case are completely analogous to those in \cite{froyshov-inequality}. However, $e_0(\mathcal{L})\geqslant e_p(\mathcal{L})$, and the following shows that we do not improve upon what is already known from Theorem \ref{thm:char0}.

\vspace{0.25cm}

\begin{prop}
    Let $p\in \Z$ be odd. Then $N_\beta^p(g)\geqslant \lceil g/2 \rceil$. Equality holds if $p$ is prime and $p>g$.
\end{prop}

\vspace{0.15cm}

\begin{proof}
The proof of the first statement is similar to that of Proposition \ref{prop:gamma}. It suffices to show that $e_p(\mathcal{L})\geqslant \lceil g/2 \rceil$ where $\mathcal{L}=\Gamma_{4g+4}$. We follow \cite[Prop.1]{froyshov-inequality}. Consider the extremal vector $w=(1,\ldots,1,0,\ldots,0)\in\mathcal{L}$ having $4\lceil g/2 \rceil$ entries equal to $1$. If $g$ is odd then $\text{Min}(w+2\mathcal{L})$ consists of $(\pm 1,\ldots,\pm 1, 0,\ldots,0)$ and $(0,\ldots,0,\pm 1,\ldots, \pm 1)$ where the number of signs is even; if $g$ is even it consists of $(\pm 1,\ldots,\pm 1, 0,\ldots,0)$ where again the number of signs is even. In either case, the signs in $\eta(\mathcal{L},w)$ are all equal, and $\eta(\mathcal{L},w)= \pm \frac{1}{2}\#\text{Min}(w+2\mathcal{L})$ is a power of 2, and in particular nonzero mod $p$. Since $w^2=4\lceil g/2\rceil$, we conclude that $e_p(\mathcal{L})\geqslant \lceil w^2/4 \rceil = \lceil g/2 \rceil$.\\

For the second statement, we follow \cite[Prop. 20]{munoz}, and use our notation from Section \ref{sec:surface}. The recursive equation defining $\zeta_{g+1}$ yields $g^2(\beta+(-1)^g)\zeta_{g-1} \equiv \zeta_{g+1} - \alpha\zeta_g$ (mod $\gamma$). Thus we have $g^2(\beta+(-1)^g8)J_{g-1}\subset J_g + (\gamma)$ where $J_g=(\zeta_g,\zeta_{g+1},\zeta_{g+2})$. Inductively, in $\mathbb{V}_g$ we obtain
    \begin{equation*}
        \prod_{r=1}^g r^2 (\beta + (-1)^r8) \; = \; \gamma\phi\label{eq:nilp1}
    \end{equation*}
for some $\phi\in \mathbb{V}_g$. Now since $p$ is prime and $p>g$, the factor $1^22^2\cdots g^2$ has an inverse mod $p$. After multiplying both sides by this inverse, and, if $g$ is odd, multiplying by $(\beta+8)$, we obtain the relation $(\beta^2-64)^{\lceil g/2 \rceil} \equiv 0$ (mod $\gamma$) within $\mathbb{V}_g\otimes \Z/p$, implying $N_\beta^p(g)\leqslant \lceil g/2 \rceil$.
\end{proof}

%% file: anotherproof.tex
%!TEX root = main.tex

The only instanton Floer theory used in the above proofs of Theorems \ref{thm:genus2} and \ref{thm:genus1} is the input from certain relations in the instanton Floer cohomology of a circle times a surface via Theorem \ref{thm:char4} and Proposition \ref{prop:mod8}; the instanton homology of homology 3-spheres is not required at all. In this Section we deduce Corollaries \ref{cor:2,5} and \ref{cor:2,3} with this latter framework at heart, with some help from Floer's exact triangle. While the two approaches complement one another, they also perhaps belong together in a more natural framework as suggested by Fr\o yshov's inequality in characteristic zero, Theorem \ref{thm:char0}; we merely scratch the surface here for $\Z/2$ and $\Z/4$ coefficients.\\

For an integer homology 3-sphere $Y$, denote by $I^\ast(Y;\F_2)$ Floer's instanton (co)homology from \cite{floer}, defined with $\F_2=\Z/2$ coefficients, and using the conventions of \cite{froyshov-equivariant}. This is a $\Z/8$-graded vector space over $\F_2$. There are elements $\delta_2\in I^4(Y;\F_2)^\ast$ and $\delta_2'\in I^1(Y;\F_2)$ defined using moduli spaces of insantons with a trivial flat limit at either end of $Y\times\R$. There is also a degree $2$ endomorphism on $I^\ast(Y;\F_2)$, denoted $v_2$, and defined using the second Stiefel-Whitney class of the $SO(3)$ basepoint fibration, analogous to how the degree 4 endomorphism $u$ is defined in \cite{froyshov-equivariant} on $I^\ast(Y;\Z)$ for certain gradings using the first Pontryagin class.\\

The elements $\delta_2\in I^4(Y;\F_2)^\ast$ and $\delta_2'\in I^1(Y;\F_2)$ are induced by (co)chains $\delta\in CI^4(Y;\F_2)^\ast$ and $\delta'\in CI^1(Y;\F_2)$ defined just as in \cite[2.1]{froyshov-equivariant}, but with $\F_2$-coefficients, which we now review. Recall that the cochain complex $CI^\ast(Y;\F_2)$ is generated by (perturbed) flat irreducible $SU(2)$ connections mod gauge. We will follow the notation of \cite{froyshov-equivariant} and write $M(\alpha,\beta)$ for the moduli space of finite-energy instantons on $\R\times Y$ with flat limit $\alpha$ at $+\infty$ and $\beta$ at $-\infty$, and with expected dimension lying in $[0,7]$. Write $\check{M}(\alpha,\beta)=M(\alpha,\beta)/\R$. The cochain $\delta'$ is then defined to be $\sum \# \check{M}(\beta,\theta)\beta$, where $\beta$ runs through the generators of $CI^1(Y;\F_2)$, and $\theta$ is the trivial connection. Similarly, $\delta \alpha = \#\check{M}(\theta,\alpha)$ for a generator $\alpha\in CI^4(Y;\F_2)$.\\

The map $v_2$ is induced by a degree $2$ cochain map $v$ on $CI^\ast(Y;\F_2)$, defined as follows. Let $\alpha$ and $\beta$ be generators such that $M(\alpha,\beta)$ is 2-dimensional. Let $\mathbb{E}_0\to M(\alpha,\beta)$ be the natural euclidean 3-plane bundle associated to a basepoint $(0,y_0)$. Choose sections $\sigma_1$ and $\sigma_2$ of $\mathbb{E}_0$ which are pulled back from the basepoint fibration over the configuration space of connections on $(-1,1)\times Y$. We arrange that $\sigma_1$ and $\sigma_2$ are linearly dependent at finitely many points, and transversely. Set
\begin{equation}
    \langle v(\beta), \alpha \rangle \; = \; \#\left\{[A]\in M(\alpha,\beta):\;\; \sigma_1([A])\in \R\cdot \sigma_2([A])\right\}.\label{eq:v2}
\end{equation}
That $v$ is a chain map, and is independent of any choices made, follows the proof of \cite[Thm. 4]{froyshov-equivariant}, except there are no trajectories that break at the reducible. Indeed, since $\dim M(\alpha,\beta)=2$, the relation $dv+vd=0$ comes from counting the ends of a 3-dimensional moduli space, cut down by two sections as above; such a moduli space has ends approaching trajectories broken at a trivial connection if its dimension is $\geqslant 5$, see \cite[\S 5.1]{donaldson-floer}. The construction of $v_2$ and its interactions with the analogous map for the third Stiefel-Whitney class of $\mathbb{E}_0$ was sketched by Fr\o yshov \cite{froyshov-regens}.

\vspace{0.25cm}

\begin{prop}\label{prop:v2}
    Let $X$ be a smooth, compact, oriented 4-manifold with negative definite lattice $\mathcal{L}=H^2(X;\Z)/\text{{\emph{Tor}}}$ and boundary an integer homology 3-sphere $Y$. If $H^\ast(X;\Z)$ has no $2$-torsion,
    \begin{equation*}
        \min\left\{j\geqslant 0: \; \delta_2 v_2^j=0\right\}\;  \geqslant \; f_2(\mathcal{L}).
    \end{equation*}
\end{prop}

\vspace{0.25cm}

The proof is an adaptation of Proposition 1 in \cite{froyshov-equivariant}, which uses the additional assumption $b_1(X)=0$. Fix $w\in H^2(X;\Z)$ descending to an extremal vector of the same name in $\mathcal{L}$, and $a\in \text{Sym}^m H_2(X;\Z)$ descending to an element of the same name in $\text{Sym}^m (\mathcal{L}^w)$ for some $m\geqslant 0$. For simplicity assume $a=a_1\cdots a_m$ where $a_i\in H_2(X;\Z)$. Form $X^+$ by attaching a cylindrical end $[0,\infty)\times Y$ to the boundary of $X$, and fix a metric $g$ on $X^+$which is a product along the end.
Consider the moduli space $\mathscr{M}_0$ of suitably perturbed instantons on $(X^+,g)$ on a $U(2)$-bundle with $c_1=w$ and relative second Chern number $k=0$ which are asymptotic to the trivial connection. Removing neighborhoods of reducibles we have a smooth moduli space $\mathscr{M}'_0$ of dimension $2|w^2|-3$. Cut this down as follows to obtain an unoriented 1-manifold:
\begin{equation*}
   \nu(x)^{|w^2|-m-2}\prod_{k=1}^m \mu(a_k) \mathscr{M}'_0
\end{equation*}
Similar to \cite{froyshov-equivariant}, counting the boundary points yields $2^{-m}\eta(\mathcal{L},w,a,m) \pmod 2$. The ends contribute the term $\delta_2 v_2^j \cdot D_X^w(a)$ where $D_X^w(a)$ is a relative Donaldson invariant $D_X^w(a)\in I^{4-4n}(Y;\F_2)$ where $n=|w^2|-m-2$. We obtain the second equality in
\begin{equation} \label{eq:relmod2reldoninv}
    \delta_2 v_2^j \cdot D_X^w(a) \; = \; \begin{cases} 0 & \text{for }0\leqslant j<n \\ 2^{-m}\eta(\mathcal{L},w,a,m)\;\;(\text{mod }2) & \text{for }j=n. \end{cases}
\end{equation}
The first equality follows from the same argument but with $k<0$, where there are no reducibles (and in some cases is true for degree reasons). The statement of Proposition \ref{prop:v2} follows for $b_1(X)=0$ from this formula and the definition of $f_2(\mathcal{L})$; the condition that $b_1(X)=0$ is then handled by surgering loops, cf. \cite[Prop.2]{froyshov-inequality}.\\

We have a similar inequality for $\Z/4$ coefficients. Here we let $u$ denote the degree 4 map defined on $CI^\ast(Y;\Z)$ as in \cite{froyshov-equivariant}, which in general is not a chain map, but satisfies $du-ud+2\delta\otimes \delta'=0$. The map $\delta u^n:CI^{4-4n}(Y;\Z)\to \Z$ is a chain map, and we denote by $\delta_4 u_4^n$ the map $I^{4-4n}(Y;\Z/4)\to \Z/4$ obtained after tensoring with $\Z/4$ and taking homology. This may depend on auxiliary choices, such as perturbation and metric; in the following assume choices are fixed.

\vspace{0.25cm}

\begin{prop}\label{prop:u4}
    Let $X$ be a smooth, compact, oriented 4-manifold with negative definite lattice $\mathcal{L}=H^2(X;\Z)/\text{{\emph{Tor}}}$ and boundary an integer homology 3-sphere $Y$. If $H^\ast(X;\Z)$ has no $2$-torsion,
    \begin{equation*}
        \min\left\{j\geqslant 0: \; \delta_4 u_4^j = 0\right\}\;  \geqslant \; f_4(\mathcal{L}).
    \end{equation*}
\end{prop}

\vspace{0.25cm}

The proof is similar to that of Proposition \ref{prop:v2}, but more directly uses the formula of \cite[Prop.1]{froyshov-equivariant}, the statement of which is the following, assuming $b_1(X)=0$; the proposition uses its mod 4 reduction. For $w\in H^2(X;\Z)$ descending to an extremal vector of the same name in $\mathcal{L}$, and $a\in \text{Sym}^m H_2(X;\Z)$ descending to an element of the same name in $\text{Sym}^m(\mathcal{L})$ for $m\geqslant 0$, there is a relative invariant $D_X^w(a)\in I^{4-4n}(Y;\Z[1/2^m])$ where $n=(|w^2|-m)/2-1$, and
\begin{equation}
    \delta u^j \cdot D_X^w(a) \; = \; \begin{cases} 0 & \text{for }0\leqslant j<n \\ 2^{-m}\#\mathcal{T}\cdot\eta(\mathcal{L},w,a,m) & \text{for }j=n. \end{cases}\label{eq:reldon}
\end{equation}
For both Propositions \ref{prop:v2} and \ref{prop:u4} we use the knowledge from Section \ref{subsec:proofs} of what kinds of classes $a$ can cut down moduli spaces when working over the appropriate coefficient ring. For example, if $a\in \text{Sym}^m(\mathcal{L}^w)$ then $D_X^w(a)$ is an element in $I^{4-4n}(Y;\Z)$, and the factor $1/2^m$ is unnecessary in the coefficient ring.\\

We expect that the left-hand sides of the inequalities of Propositions \ref{prop:v2} and \ref{prop:u4} can be replaced by more natural quantities. For example, the first of these should be a weaker form of a general inequality involving Fr\o yshov's homology cobordism invariant $q_2$ mentioned in the introduction. Similarly, the second is related to Fr\o yshov's framework as developed in \cite{froyshov-equivariant}, but with $\Z/4$ coefficients. We are now in a position to give an alternative proof of Corollary \ref{cor:2,3}.

\vspace{0.25cm}

\begin{proof}[Another proof of Corollary \ref{cor:2,3}] 
Let $Y=\Sigma(2,3,5)$. It is well-known that $CI^\ast(Y;\Z)$ is generated by two flat $SU(2)$ connections in degrees $0$ and $4$. The differential on $CI^\ast(Y;\Z)$ is zero, and hence $u$ is a chain map, and induces a map on $I^\ast(Y;\Z)$ which we also call $u$. By \cite[Prop.2]{froyshov-equivariant}, $\delta u$ is divisible by $8$, and in particular $\delta_4u_4 \equiv \delta u $ (mod 4) vanishes. The degree two map $v_2$ on $I^\ast(Y;\Z/2)$ is zero for grading reasons. Thus the left-hand sides of the inequalities in Propositions \ref{prop:v2} and \ref{prop:u4} are equal to 1, and the result follows from Lemmas \ref{lemma:indecomp} and \ref{lemma:e8}.
\end{proof}

\vspace{0.25cm}

Alternatively, we can also formulate the mod $8$ analogue of Proposition \ref{prop:u4} together with the computation $f_8(\Gamma_{12})=2$, and use this with Proposition \ref{prop:u4} and Lemma \ref{lemma:f4}. However, the above proof illustrates that constraints from inequalities arising from the mod $2$ and mod $4$ coefficient cases are sufficient. Further, the use of only Lemmas \ref{lemma:indecomp} and \ref{lemma:e8} shows that there is a minimal amount of algebra needed.\\

The computation $\delta u \equiv 0$ (mod 8) in the proof of Corollary \ref{cor:2,3} is computed in \cite{froyshov-equivariant} via basic gluing formulae for relative Donaldson invariants, using an embedding of the negative definite $E_8$ plumbing into a $K3$ surface. The same procedure may be attempted for $\Sigma(2,5,9)$, the boundary of a negative definite plumbing with intersection form $-\Gamma_{12}$ which itself embeds in the elliptic surface $E(3)$, as follows from \cite[Sec.2]{fs-can}, and builds on the construction explained at the beginning of Section \ref{sec:genus2}. However, we can obtain the congruence $\delta u \equiv 0$ (mod 8) for the Poincar\'{e} sphere by another method, which will also lead to another proof of Corollary \ref{cor:2,5} for $\Sigma(2,5,9)$, without reverting to gluing formulae for Donaldson invariants. We proceed to explain this.\\

As in the above proof, let $Y=\Sigma(2,3,5)$, and denote by $P$ the non-trivial $SO(3)$-bundle over $0$-surgery on the $(2,3)$-torus knot. Then for any coefficient ring we have the long exact sequence
\begin{equation}
 \cdots \;\;I^\ast(S^3) \xrightarrow{}   I^\ast(Y)   \xrightarrow{W_\ast}       I^\ast(P)              \xrightarrow{}         I^\ast(S^3)\;\; \cdots \label{eq:longexact1}
\end{equation}
This is Floer's exact triangle \cite{floer-triangle, bd}. The map $W_\ast$ is induced by a surgery 2-handle cobordism $W:Y\to Y_0$. Because the instanton cohomology of the 3-sphere vanishes, $W_\ast$ is an isomorphism. For the non-trivial bundle $P$, the map $u$ is also defined on instanton cohomology. The map $W_\ast$ does not commute with $u$; in fact $W_\ast u - u W_\ast = 2 \delta\otimes \delta'_W$ where $\delta_{W}'$ counts isolated instantons on $W$ with trivial limit at $Y$. (This follows from a version of \cite[Theorem 6]{froyshov-equivariant}.) From this it follows, however, that $W_\ast u = u W_\ast$ on $I^0(Y;\Z)$. Thus to show that $\delta u \equiv 0$ (mod 8) on $I^0(Y;\Z)$ it suffices to show that $u\equiv 0$ (mod 8) on $I^\ast(P;\Z)$.\\

The (2,3) torus knot has genus 1. Consequently, there is a genus 1 surface embedded in the 0-surgery over which the bundle $P$ restricts non-trivially; this is formed by capping off a Seifert surface in the complement of the surgery neighborhood with a disk glued in from $0$-surgery. Following \cite[\S 6]{froyshov-equivariant} we stretch along a link of this surface in $\R$ cross the 0-surgery diffeomorphic to a 3-torus $T^3$ to conclude that $u$ factors through the corresponding map on $\mathbb{V}_1'$. However, on this latter group, $u=\beta'\equiv 0$ (mod 8), establishing the claim. We note that essentially the same argument shows that $\delta u\equiv 0$ (mod 8) for the family of Brieskorn spheres $\Sigma(2,3,6k\pm 1)$, and so we obtain alternative proofs for Corollaries \ref{cor:2,3,11} and \ref{cor:2,3,12n+5} as well.

\vspace{0.25cm}

\begin{proof}[Another proof of Corollary \ref{cor:2,5}] 
Let $Y=\Sigma(2,5,9)$. The exact sequence (\ref{eq:longexact1}) now applies to surgery on the (2,5) torus knot. As for $\Sigma(2,3,5)$, the Floer complex for $Y$ has zero differential and $u$ is chain map. Again, although $u$ and $W_\ast$ do not commute in general, they do on $I^0(Y;\Z)$. Furthermore, $v_2$ and the mod 2 reduction of $W_\ast$ commute. Next, the $(2,5)$ torus knot is of genus 2, and $I^\ast(P;\Z)$ has $u\equiv 0 $ (mod 4) and $v_2^2 \equiv 0 $ (mod 2) since $\beta'\equiv 0$ (mod 4) and $(\alpha')^2\equiv 0$ (mod 2) within $\mathbb{V}'_2$. Now the left-hand sides of the inequalities in Propositions \ref{prop:v2} and \ref{prop:u4} are $2$ and $1$, respectively, and with Lemma \ref{lemma:f4} the result follows.
\end{proof}

%% file: e72.tex
%!TEX root = main.tex

The root lattice $\E_7$ is the subset of $\frac{1}{2}\Z^8$ consisting of vectors $x=(x_1,\ldots,x_8)$ with $\sum x_i=0$ and all $x_i$ in one of $\Z^8$ or $\frac{1}{2}+\Z^8$. The positive definite unimodular lattice $E_7^2$ is defined by
\[
    E_7^2 \; = \; \E_7 \oplus \E_7 \cup \left(g + \E_7\oplus \E_7\right),
\]
\[
    g=\left((\tfrac{3}{4}^2, -\tfrac{1}{4}^6),(\tfrac{3}{4}^2, -\tfrac{1}{4}^6)\right)\in \E_7^\ast
\]
We note that $\E_7^\ast/\E_7$ is cyclic of order 2 generated by $[g]$. In this section we show

\vspace{0.25cm}

\begin{prop}
	$e_0(E^7_2)=1$, $f_2(E_2^7)=2$ and $f_4(E^7_2)=2$.
\end{prop}

\vspace{0.25cm}

These computations show that even if the mod 2 inequality \eqref{eq:mod2speculation} were true in general, it would not be sufficient to prove Theorem \ref{thm:genus2}. In the course of the proof to follow we leave some of the computations to the reader.

\vspace{0.25cm}

\begin{proof}
We need to understand the index two cosets of $\mathcal{L}$ and their extremal vectors. We divide the cosets into two types: those in the image of the inclusion-induced map
\[
	\pi:\frac{\E_7\oplus \E_7}{2\left(\E_7\oplus \E_7\right)} \longrightarrow \frac{\mathcal{L}}{2\mathcal{L}}
\]
and those that are not. To better understand the former case, we list the index two cosets of $\E_7$. These are easily found by hand, and are also listed in \cite[p.169]{conwaysloane}. First define
\[
x\; =\;  (1,-1,0^6), \qquad y \; = \;  (1^2,-1^2,0^4), \qquad z \; = \; (\tfrac{3}{2}^2,-\tfrac{1}{2}^6).
\]
Note that $x^2 = 2$, $y^2=4$ and $z^2=6$. Consider the cosets $w+2\E_7$ for $w\in\{0,x,y,z\}$. After applying automorphisms of $\E_7$ to these we obtain all cosets in $\E_7/2\E_7$. There are $63$ cosets in the orbit of $x+2\E_7$, each represented by a vector of square 2, unique up to sign, and there are similarly $63$ cosets in the orbit of $y+2\E_7$, each having 12 square 4 vectors. There are only two other cosets, represented by $0$ and $z$, which are fixed under the action of the automorphism group. Thus the total number of cosets is $1+63+63+1=2^7$, as expected.\\

The cosets in $\mathcal{L}$ that lie in the image of $\pi$ are therefore represented by $(u,v)$ for $u,v\in\{0,x,y,z\}$ and some cosets obtained from these by applying automorphisms. The case $(z,z)$ can be ignored; indeed, $(z,z)=2g$, so this vector represents the zero coset. Next we note $(y,z)-2g=(t,0)$ where $t=(-1/2^4,1/2^4)$ has square 2, and $(y,y)-2g =(t,t)$, a vector of square 4. Similarly, $(x,z)$ is mod 2 equivalent to a vector of square 4 supported in $\E_7\oplus 0$. Thus by symmetry, when maximizing over the data defining $e_0(\mathcal{L})$, $f_2(\mathcal{L})$ and $f_4(\mathcal{L})$ which has $w$ extremal and $w+2\mathcal{L}$ contained in the image of $\pi$, we may restrict our attention to $w$ being among $(x,0)$, $(y,0)$, $(z,0)$, $(x,x)$ and $(x,y)$.\\

Now we consider cosets not contained in the image of $\pi$. We claim that upon defining 
\[
        a \; = \; (\tfrac{3}{4}^2,-\tfrac{1}{4}^6), \qquad b \; = \;  (\tfrac{3}{4}^3,-\tfrac{5}{4},-\tfrac{1}{4}^4), \qquad c \; = \; (\tfrac{7}{4}^1,-\tfrac{1}{4}^7),
\]
all elements in $\mathcal{L}/2\mathcal{L}- \text{im}(\pi)$ are obtained from the cosets represented by $(a,a)$, $(a,b)$, $(a,c)$, $(b,b)$, $(b,c)$ and $(c,c)$ after perhaps applying an automorphism of $\mathcal{L}$. The claim is verified by counting. First note that $\pi$ has 1-dimensional kernel. Indeed, a basis for $E_7^2$ is given by the rows of the matrix:
\[\arraycolsep=2pt\def\arraystretch{1}
\left[\begin{array}{rrrrrrrrrrrrrrrr}
\frac{1}{4} & \frac{1}{4} & \frac{1}{4} & \frac{1}{4} &\frac{1}{4} & \frac{1}{4} & -\frac{3}{4} & -\frac{3}{4} &  \phantom{-}\frac{1}{4} &  \phantom{-}\frac{1}{4} & \frac{1}{4} & \frac{1}{4} & \frac{1}{4} & \frac{1}{4} & -\frac{3}{4} & -\frac{3}{4} \\
\frac{1}{2} & \frac{1}{2} & \frac{1}{2} & \frac{1}{2} & -\frac{1}{2} & -\frac{1}{2} & -\frac{1}{2} & -\frac{1}{2} & 0 & 0 & 0 & 0 & 0 & 0 & 0 & 0 \\
1 & -1 & 0 & 0 & 0 & 0 & 0 & 0 & 0 & 0 & 0 & 0 & 0 & 0 & 0 & 0 \\
0 & 1 & -1 & 0 & 0 & 0 & 0 & 0 & 0 & 0 & 0 & 0 & 0 & 0 & 0 & 0 \\
0 & 0 & 1 & -1 & 0 & 0 & 0 & 0 & 0 & 0 & 0 & 0 & 0 & 0 & 0 & 0 \\
0 & 0 & 0 & 1 & -1 & 0 & 0 & 0 & 0 & 0 & 0 & 0 & 0 & 0 & 0 & 0 \\
0 & 0 & 0 & 0 & 1 & -1 & 0 & 0 & 0 & 0 & 0 & 0 & 0 & 0 & 0 & 0 \\
0 & 0 & 0 & 0 & 0 & 1 & -1 & 0 & 0 & 0 & 0 & 0 & 0 & 0 & 0 & 0 \\
0 & 0 & 0 & 0 & 0 & 0 & 0 & 0 & 0 & 1 & -1 & 0 & 0 & 0 & 0 & 0 \\
0 & 0 & 0 & 0 & 0 & 0 & 0 & 0 & 0 & 0 & 1 & -1 & 0 & 0 & 0 & 0 \\
0 & 0 & 0 & 0 & 0 & 0 & 0 & 0 & 0 & 0 & 0 & 1 & -1 & 0 & 0 & 0 \\
0 & 0 & 0 & 0 & 0 & 0 & 0 & 0 & 0 & 0 & 0 & 0 & 1 & -1 & 0 & 0 \\
0 & 0 & 0 & 0 & 0 & 0 & 0 & 0 & 0 & 0 & 0 & 0 & 0 & 1 & -1 & 0 \\
0 & 0 & 0 & 0 & 0 & 0 & 0 & 0 & 0 & 0 & 0 & 0 & 0 & 0 & 1 & -1
\end{array}\right]
\]
Every row in the matrix except for the first lies in $\E_7\oplus \E_7$, and the first row is equivalent modulo $2\mathcal{L}$ to the vector $g\in\mathcal{L}$. Thus every coset not in the image of $\pi$ is of the form $[g]+[w]$ where $[w]\in\text{im}(\pi)$, and there are $2^{13}=8192$ such cosets.\\

We consider orbits of the automorphism group acting on $[w]=w+2\mathcal{L}$ as $w$ varies through the above representatives. Let $G$ be the subgroup of $\text{Aut}(\mathcal{L})$ generated by automorphisms that permute the first 8 or last 8 coordinates, the automorphism that swaps the first and last 8 coordinates, and the automorphism $\sigma$ that negates the first 8 coordinates. First consider
\[
    w \; = \; (a,a)\; = \; \left(\left( \phantom{}\tfrac{3}{4},\phantom{}\tfrac{3}{4},-\tfrac{1}{4},-\tfrac{1}{4},-\tfrac{1}{4},-\tfrac{1}{4},-\tfrac{1}{4},-\tfrac{1}{4}\right),\left( \phantom{}\tfrac{3}{4},\phantom{}\tfrac{3}{4},-\tfrac{1}{4},-\tfrac{1}{4},-\tfrac{1}{4},-\tfrac{1}{4},-\tfrac{1}{4},-\tfrac{1}{4}\right)\right)
\]
Then $G\cdot w$ consists of vectors obtained from $w$ by permuting the ``$3/4$'' terms within each $\E_7$-factor and changing signs on each $\E_7$-factor. Thus $\# G\cdot w = 4\cdot {8\choose 2}\cdot {8\choose 2}=3136$. The only mod 2 congruences among $v\in G\cdot w$ are $v\equiv -v$, and so $\# G\cdot [w] = \frac{1}{2}\# G\cdot w = 1568$. Next, consider
\[
    w \; = \; (a,b)\; = \; \left(\left( \phantom{}\tfrac{3}{4},\phantom{}\tfrac{3}{4},-\tfrac{1}{4},-\tfrac{1}{4},-\tfrac{1}{4},-\tfrac{1}{4},-\tfrac{1}{4},-\tfrac{1}{4}\right),\left( \phantom{}\tfrac{3}{4},\phantom{}\tfrac{3}{4},\tfrac{3}{4},-\tfrac{5}{4},-\tfrac{1}{4},-\tfrac{1}{4},-\tfrac{1}{4},-\tfrac{1}{4}\right)\right)
\]
We compute $\# G\cdot w  = 4\cdot 2\cdot {8 \choose 2}\cdot 4{8 \choose 4}$. However, the stabilizer for $G$ acting on $[w]$ consists of $\sigma$, the automorphism swapping the first and last 4 coordinates of the second $\E_7$-factor, and any permutation preserving the first 12 coordinates. Taking this into account, we compute $\# G\cdot [w]  = 2\cdot {8 \choose 2}\cdot {8 \choose 4}=3920$. We proceed in this manner to find that $\# G\cdot [w]$ for $w$ among $(a,a)$, $(a,b)$, $(a,c)$, $(b,b)$, $(b,c)$ and $(c,c)$ is equal to $1568$, $3920$, $112$, $2450$, $140$ and $2$, respectively. These add up to $8192$, and this verifies the claim stated in the previous paragraph.\\

In summary, when maximizing over the data defining $e_0(\mathcal{L})$, $f_2(\mathcal{L})$ and $f_4(\mathcal{L})$ we may restrict our attention to computing $\eta(\mathcal{L},w,a,m)$ for which $w$ is in the following table.\\

\begin{table}[h!]
  \centering
\setlength\tabcolsep{.15cm}
\renewcommand{\arraystretch}{1.2}
  \begin{tabular}{cccccccccccc}
    $w$ &  $(x,0)$ & $(a,a)$ & $(x,x)$ & $(y,0)$ & $(a,b)$ & $(a,c)$ & $(x,y)$ & $(z,0)$ & $(b,b)$ & $(b,c)$ & $(c,c)$\\
    \hline
    $w^2$  & $2$ & $3$ & $4$ & $4$ & $5$ & $5$ & $6$ & $6$ & $7$ & $7$ & $7$   \\ 
  \end{tabular}
\end{table}
 
In each case $w$ is extremal. We next claim the following, where in each case $w$ runs over the vectors in the table: (i) $\eta(\mathcal{L},w)=0$ for $w\in\{(x,y),(z,0)\}$ and $\eta(\mathcal{L},w,e,1)=0$ for $w\in\{(b,b),(b,c),(c,c)\}$ as $e$ runs over a basis for $\mathcal{L}$; and (ii) $2^{-m}\eta(\mathcal{L},w,e_1\cdots e_m,m)\equiv 0$ (mod 2) for each $m\geqslant 0$ with $w^2-m\geqslant 4$, where $e_1,\ldots, e_m$ are arbitrary elements of a basis for $\mathcal{L}^w$. It is straightforward to verify these claims by computer once one knows $\text{Min}(w+2\mathcal{L})$ for each $w$ above. For example, if $w=(c,c)$, this set consists of the $64$ vectors obtained by permuting the placement of the $7/4$ terms within each $\E_7$-summand. Claim (i) implies $e_0(\mathcal{L})=1$ and $f_4(\mathcal{L})=2$. More precisely, for the latter, it establishes that $f_4(\mathcal{L})\leqslant 2$, and Lemma \ref{lemma:indecomp} implies equality. Claim (ii) implies $f_2(\mathcal{L})=2$. This completes the proof.
\end{proof}

\vspace{0.25cm}

The following result collects the lattices that occur in Elkies' List, Table \ref{tab:table1} above, under the constraint $f_2(\mathcal{L})\leqslant 2$. In terms of our topological tools, it combines the restrictions of having $d$-invariant 2 and the inequality of Proposition \ref{prop:v2} having left-hand side equal to 2. We do not know if the lattice $E_7^2$ ever occurs under the hypotheses given.

\vspace{0.25cm}

\begin{prop}\label{prop:further1}
    Suppose $Y$ is an integer homology 3-sphere with $d(Y)=2$ and $\delta_2 v_2 = 0$. If a smooth, compact, oriented 4-manifold with no 2-torsion in its homology has boundary $Y$ and reduced negative-definite non-diagonal intersection form $\mathcal{L}$, then $-\mathcal{L}$ is one of $E_8$, $\Gamma_{12}$, $E_7^2$.
\end{prop}

\vspace{0.15cm}

\begin{proof}
    By Proposition \ref{prop:v2} it suffices to show that $m(\mathcal{L})\geqslant 3$ for the lattices in Table \ref{tab:table1} other than the three given. As in the proof of Lemma \ref{lemma:e8}, if the the root lattice $R\subset \mathcal{L}$ contains $\A_n$ for $n\geqslant 3$, then $w=(1,1,-1,-1,0,\ldots,0)\in \A_n$ shows that $m(\mathcal{L})\geqslant 3$. This leaves $D_8^2$, $D_6^3$, $D_4^5$, $A_1^{22}$, $O_{23}$ in Table \ref{tab:table1}. The following descriptions of the first three of these lattices are from \cite{cs-det1}.\\
    
    Suppose $\mathcal{L}=D_8^2$. This lattice is generated by $\Ds_8\oplus \Ds_8$ along with $g_1=((1/2^8), (-1^1,0^7))$ and $g_2=((-1^1,0^7), (1/2^8))$. As before, superscripts denote repeated entries. Then $w=g_1+g_2$ is extremal of square 4, and $\text{Min}(w+2\mathcal{L})=\{w,-w\}$, implying $m(\mathcal{L})\geqslant w^2-1 =3$.\\
    
    Next, suppose $\mathcal{L}=D_6^3$. Then $\mathcal{L}$ is generated by the root lattice $\Ds_6\oplus \Ds_6\oplus \Ds_6$ and $g_1=(0,(1/2^6), (1/2^5,-1/2^1))$, $g_2=( (1/2^5,-1/2^1),0,(1/2^6))$ and $g_3 = ((1/2^6), (1/2^5,-1/2^1),0)$. Then $w=g_1-g_2$ is extremal of square 4 and as before $m(\mathcal{L})\geqslant 3$.\\
    
    Next, suppose $\mathcal{L}=D_4^5$. Then $\mathcal{L}$ is generated by $\Ds_4\oplus \Ds_4 \oplus \Ds_4\oplus \Ds_4 \oplus \Ds_4$ along with $g=((1/2^4)^5)$ and $g_1=(0,(0^3,1^1),(1/2^3,-1/2^1),(1/2^3,-1/2^1),(0^3,1^1))$, and cyclic permutations $g_2,g_3,g_4,g_5$ of $g_1$. Then $w=g$ is extremal of square 5 with $\text{Min}(w+2\mathcal{L})=\{w,-w\}$, implying $m(\mathcal{L})\geqslant  4$.\\
    
    Now suppose $\mathcal{L}=A_1^{22}$. We sketch the construction of this lattice following Construction A of \cite[Ch.7]{conwaysloane} using the shortened Golay code $\mathsf{C}_{22}$. Let $\mathsf{S}$ be the subspace of the Golay code $\mathsf{C}_{24}$ consisting of vectors with first two coordinates 00 or 11, where $\mathsf{C}_{24}\subset \F_2^{24}$ is spanned by the rows of Fig. 3.4 in \cite[p.84]{conwaysloane}. Then $\mathsf{C}_{22}$ is the subspace of $\F_2^{22}$ obtained by projecting $\mathsf{S}$ onto the last 22 coordinates, and $A_1^{22}$ is the subset of $\R^{22}$ consisting of vectors $\vec{x}/\sqrt{2}$ such that $\vec{x}$ (mod 2) lies in $\mathsf{C}_{22}$. Let $\vec{v}\in \{0,1\}^{22}$ descend to the code word $\vec{v}$ (mod 2) in $\mathsf{C}_{22}$ with 10 entries equal to $1$, obtained by summing the first 11 rows of Fig. 3.4 in \cite[p.84]{conwaysloane}, ignoring the first two coordinates. Then $w=\vec{v}/\sqrt{2}\in A_1^{22}$ is extremal of square 5. It is straightforward to verify that any extremal vector equivalent to $w$ is of the form $w+2r$ for a root $r$. The roots are the elements with one nonzero entry equal to $\pm 2/\sqrt{2}$. From this we obtain $\text{Min}(w+2\mathcal{L})=\{w,-w\}$, implying $m(\mathcal{L})\geqslant 4$.\\
    
    Finally, suppose $\mathcal{L}=O_{23}$, the shorter Leech lattice. Let $w$ be any vector of square 5; such a vector exists by inspecting the theta-series of $O_{23}$, see (7) in \cite[p.443]{conwaysloane}. Using that $O_{23}$ has no roots, $\text{Min}(w+2\mathcal{L})=\{w,-w\}$, and $m(\mathcal{L})\geqslant 4$.
\end{proof}

\vspace{0.25cm}